\documentclass[12pt, a4paper, oneside, reqno]{amsart}

\usepackage[width=15cm, height=22cm, centering]{geometry} 
\usepackage{hyperref}				
\usepackage{float}
\usepackage{graphicx}
\usepackage{framed}

\numberwithin{equation}{section}			
\numberwithin{figure}{section}				

\theoremstyle{plain}
\newtheorem{theorem1}{Theorem}[section]
\newtheorem{corollary1}[theorem1]{Corollary}
\newtheorem{lemma1}[theorem1]{Lemma}

\theoremstyle{definition}
\newtheorem{definition1}[theorem1]{Definition}
\newtheorem{example}[theorem1]{Example}

\theoremstyle{remark}
\newtheorem{remark1}[theorem1]{Remark}
\newtheorem{remarks1}[theorem1]{Remarks}

\newenvironment{theorem}{\begin{framed}\begin{theorem1}}{\end{theorem1}\end{framed}}
\newenvironment{corollary}{\begin{framed}\begin{corollary1}}{\end{corollary1}\end{framed}}

\newenvironment{definition}{\begin{framed}\begin{definition1}}{\end{definition1}\end{framed}}
\newenvironment{remark}{\begin{framed}\begin{remark1}}{\end{remark1}\end{framed}}

\DeclareMathOperator{\sn}{sn}
\DeclareMathOperator{\cn}{cn}
\DeclareMathOperator{\dn}{dn}
\DeclareMathOperator{\en}{en}
\DeclareMathOperator{\arccosh}{arccosh}

\DeclareMathOperator{\am}{am}

\newcommand{\N}{\mathbb{N}}                  

\newcommand{\R}{\mathbb{R}}

\newcommand{\cj}[1]{{\overline{#1}}}            
\newcommand*\dif{\mathop{}\!\textnormal{\slshape d}}

\begin{document}

\title{The {\sc Liouville} line element and the energy of the diagonals}

\author{\sc C{\v a}lin--{\c S}erban B{\v a}rbat}

\date{\today}


\begin{abstract}
In this work I show that in each rectangle formed by the parameter curves on a Liouville surface the energies of the main diagonals are equal. This result extends naturally to $n$-dimensional Liouville manifolds.
\end{abstract}

{\centering\maketitle}

\tableofcontents

\section{Introduction}

In this work I show that in each rectangle formed by the parameter lines on a {\sc Liouville} surface the diagonals have the same energy (see the main theorem \ref{main_thm}).
This is valid, when the surface has what I call an orthogonal {\sc Liouville} line element (including all special cases, like the isothermal one for $U_2 = U_1$ and $V_2 = V_1$ which is known as {\sc Liouville} line element in the literature) of the form:
\begin{align*}
{\dif s}^2 = ( U_1(u) + V_1(v) ) {\dif u}^2 + ( U_2(u) + V_2(v) ) {\dif v}^2 
\end{align*}
The diagonals are, in general, not geodesics on the surface; they are the image curves of the diagonals in a rectangle formed by two pairs of parameter lines in the definition domain. If the line element is isothermal 
then the parametrization of the surface is conformal and preserves angles between curves and therefore the diagonals on the surface are isogonal trajectories with respect to the parameter curves of the surface.

\subsection{Plane isothermal {\sc Liouville} maps}

In the plane there are only four distinct isothermal {\sc Liouville} maps (up to similarity transforms), see figure \ref{fig:plane_li_2}.
\begin{figure}[hbt]
\centering
\includegraphics[scale=0.5]{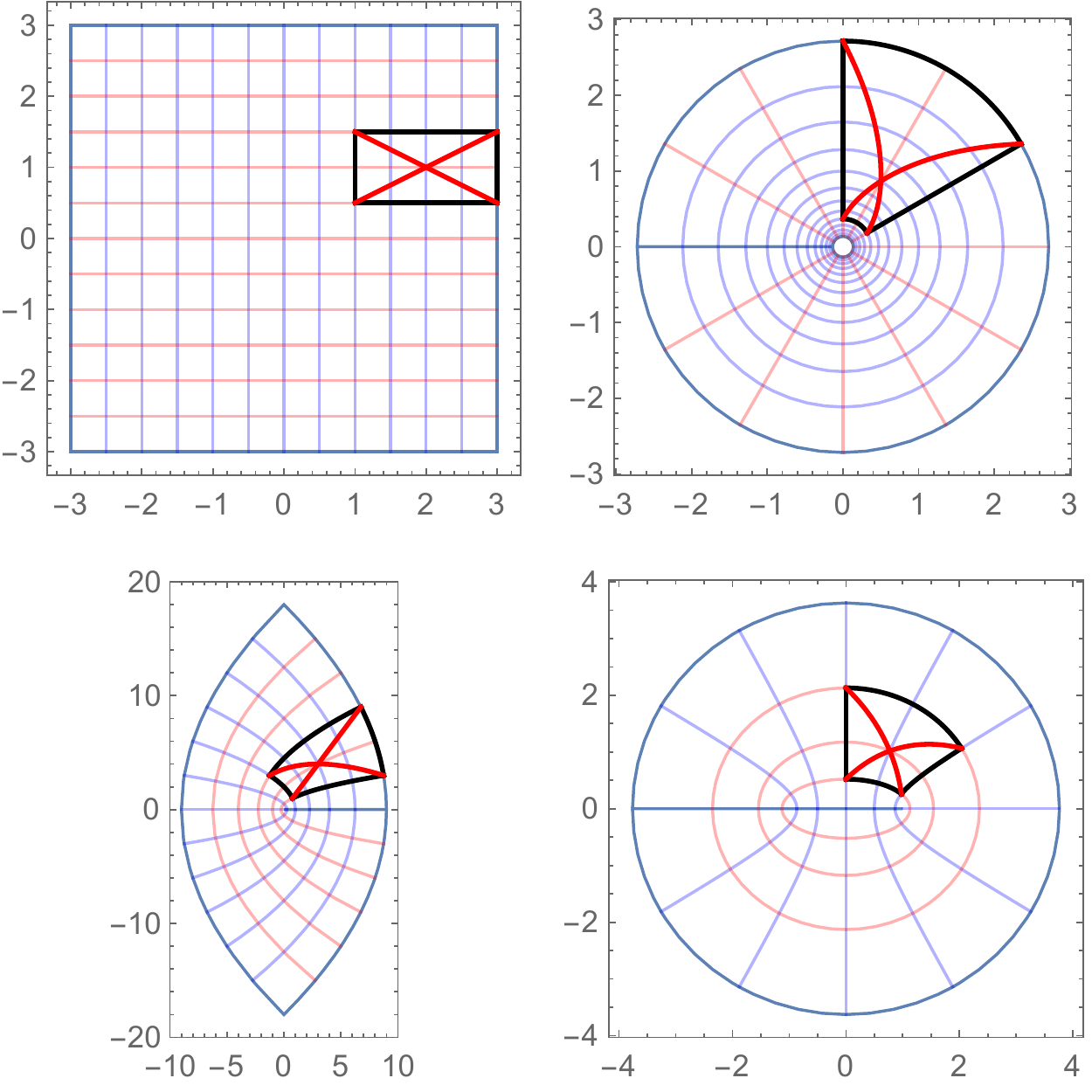} 
\caption{Isothermal {\sc Liouville} coordinates on the plane: the two diagonals in each rectangle have the same energy}
\label{fig:plane_li_2}
\end{figure}
The diagonals of each rectangle formed by parameter lines in the domain of definition are mapped to isogonal trajectories with respect to the parameter lines in the image plane because the map is isothermal and therefore conformal, see figure \ref{fig:plane_li_2}. 

Because the line element of these maps is isothermal {\sc Liouville}, the diagonals in each rectangle formed by pairs of parameter lines (belonging each to the two different families of parameter lines) have the same energy (see theorem \ref{main_thm_1} for the proof).

If we take the straight line diagonals -- these are geodesics in the plane -- then we have the theorem of {\sc Ivory}: 
\begin{theorem}[{\sc Ivory}]
The geodesic diagonals in a rectangle formed by parameter lines have the same length if and only if the map has a {\sc Stäckel} line element.
\end{theorem}
\begin{remark}
This theorem was proved by {\sc W. Blaschke} \cite{bla} and his student {\sc K. Zwirner} \cite{zwi} for the $2$ and $3$ dimensional case and it is valid in higher dimensions too. A proof is also found in \cite{Iz1} and \cite{Iz2} where they say that {\sc A. Thimm} filled a gap in {\sc W. Blaschke}'s proof.
\end{remark}
The isothermal {\sc Liouville} map is of {\sc Stäckel} type and therefore the straight line diagonals have the same length, see figure \ref{fig:plane_li_3}.
\begin{figure}[H]
\centering
\includegraphics[scale=0.5]{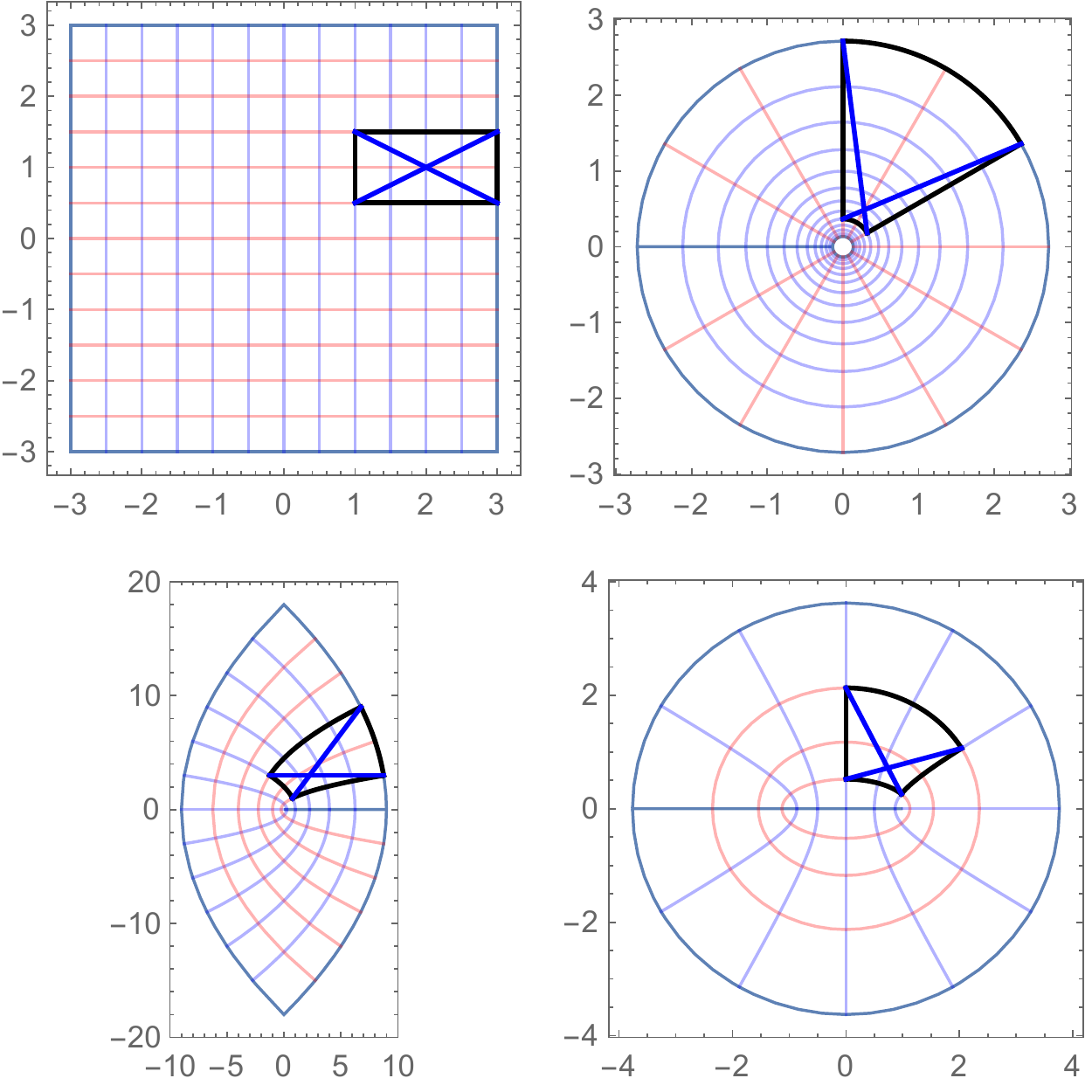} 
\caption{Isothermal {\sc Liouville} coordinates on the plane: the two (straight line) geodesic diagonals in each rectangle have the same length, according to {\sc Ivory}'s theorem}
\label{fig:plane_li_3}
\end{figure}

If we take a parameter line rectangle in one of the four isothermal {\sc Liouville} coordinates, then we can prove a statement about the energies of the discrete diagonals of the rectangle, in the sense of discrete differential geometry. We take the common interval of definition of the diagonals and partition it uniformly in $k$ pieces. Then we construct polygons approximating the diagonals. We can show that for each $k \in \N$ the two polygons have the same energy, see figure \ref{fig:plane_li_4}.
\begin{figure}[H]
\centering
\includegraphics[scale=0.5]{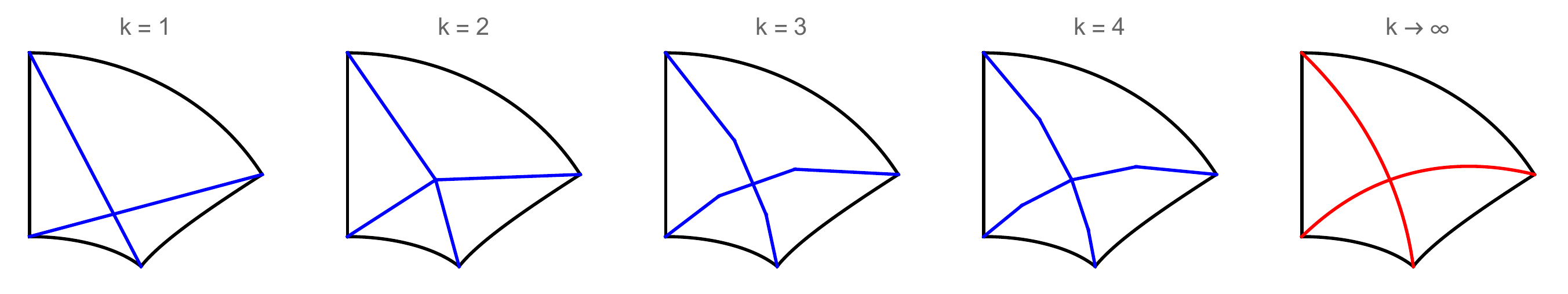} 
\caption{For each $k$ the two (discrete) diagonals have the same energy}
\label{fig:plane_li_4}
\end{figure}
Here $k = 1$ corresponds to the theorem of {\sc Ivory} in the plane because the two diagonals are geodesics and then their length and energy are related by the {\sc Schwarz} inequality. They not only have the same energy (the energy of a segment equals the squared length) in this case, but also the same length.

The limiting case $k \to \infty$ corresponds to the main theorem \ref{main_thm} of this work. All this can be viewed as a generalization of {\sc Ivory}'s theorem (see \cite{bar1}).

\subsection{Plane orthogonal {\sc Liouville} maps}

An example for a plane orthogonal {\sc Liouville} map in $U_1(u)$ and $V_2(v)$ which is also of {\sc Stäckel} type can be seen in figure \ref{fig:plane_li_5}.
\begin{figure}[H]
\centering
\includegraphics[scale=0.5]{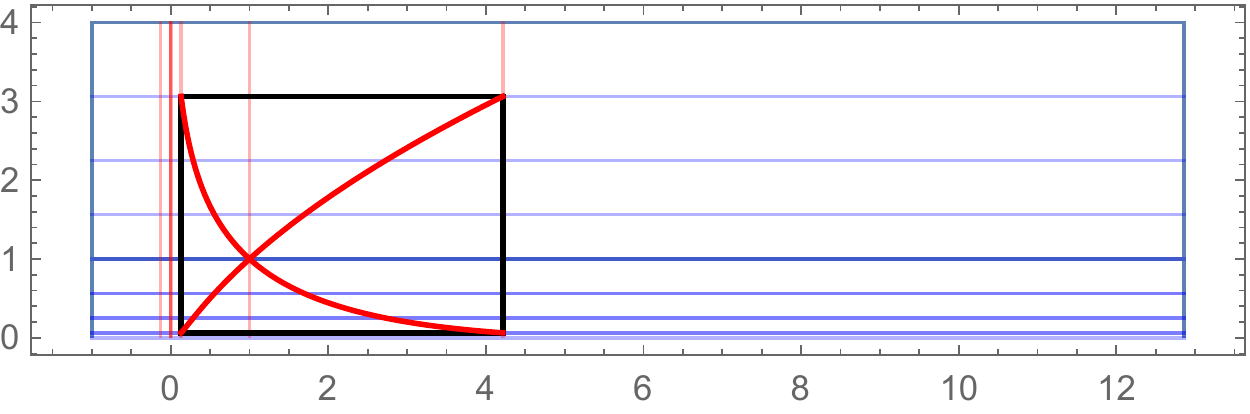} 
\caption{Orthogonal {\sc Liouville} map}
\label{fig:plane_li_5}
\end{figure}
The statement about the energies of the discrete diagonals also holds in this case, see figure \ref{fig:plane_li_6}.
\begin{figure}[H]
\centering
\includegraphics[scale=0.5]{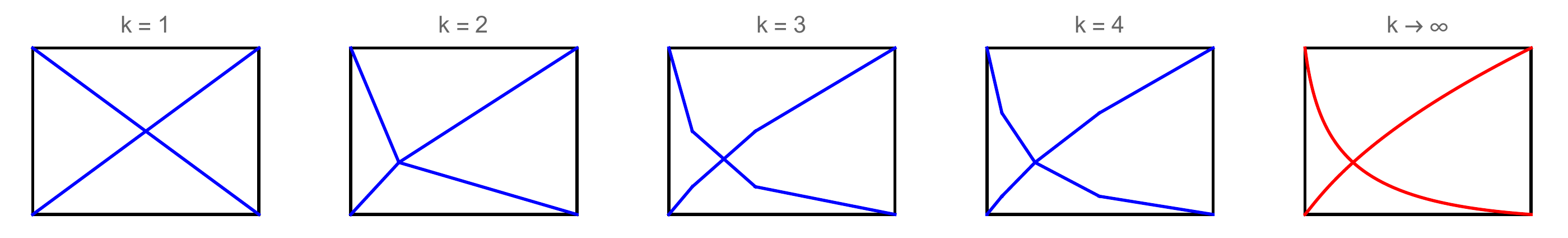} 
\caption{For each $k$ the two (discrete) diagonals have the same energy}
\label{fig:plane_li_6}
\end{figure}
For $k=1$ we have the theorem of {\sc Ivory} and for each value of $k$ the two diagonals have the same energy. The generalization of {\sc Ivory}'s theorem is valid in this case.

I conjecture that this generalization of {\sc Ivory}'s theorem also holds in parameter line rectangles on {\sc Liouville} surfaces. Instead of straight lines, there we have to consider polygons formed with segments of geodesics on the surface.

\section{Energy of the diagonals}

\subsection{Definitions}

The length $L(p)$ and the energy $E(p)$ of a curve $p : [a, b] \rightarrow M$ in a Riemannian manifold $M$ are given by the following expressions (see \cite{doc2}, (Chapter 9, p. 194)) and they are related by the {\sc Schwarz} inequality (with equality if and only if $\left|\dot{p}(t)\right|$ is constant, that means $p(t)$ is parametrized proportionally to arc length):
\begin{align*} 
L(p)&=\int_a^b{\left|\dot{p}(t)\right|  }{\dif t} = \int_a^b{\sqrt{\dot{p}(t) \cdot \dot{p}(t)}}{\dif t}\\
E(p)&=\int_a^b{\left|\dot{p}(t)\right|^2}{\dif t} = \int_a^b{      \dot{p}(t) \cdot \dot{p}(t) }{\dif t}\\
L^2(p) &\le (b - a) E(p)
\end{align*}
where $\dot{p}(t)$ is the tangent vector of the curve and $\_ \cdot \_$ is the scalar product in $TM$. 

\begin{remark}
Let's assume that $\left|\dot{p}(t)\right| = q =\text{const.}$ Then we have:
\begin{align*} 
L^2(p) = \left(\int_a^b q {\dif t}\right)^2 = (b-a)^2 q^2
= (b - a) \int_a^b q^2 {\dif t} = (b - a) E(p)
\end{align*}
The length and energy have the expressions:
\begin{align*} 
L(p) = (b-a) q \quad \text{and} \quad E(p) = (b - a) q^2
\end{align*}
If $0 < q < 1$ then $L(p) > E(p)$. If $1 < q$ then $L(p) < E(p)$.

For an arc length parametrized geodesic $\left|\dot{\gamma}(t)\right| = q = 1$ we have equality:
\begin{align*} 
L(\gamma) = (b-a) = E(\gamma)
\end{align*}
Thus the geodesic $\gamma(t)$ is (at the same time!) energy and length minimizing.
\end{remark}
If the manifold $M$ is the {\sc Euclid}ean plane, we can construct a discretization as follows: we partition the interval $[a, b]$ uniformly in $m$ pieces:
\begin{align*} 
a < a + 1 \frac{b - a}{m} <  a + 2 \frac{b - a}{m} < \cdots <  a + (m - 1) \frac{b - a}{m} <  b 
\end{align*}
The discrete length is then given by:
\begin{align*} 
L(p, m) = \sum_{k = 1}^{m} \left| p \left(a + k \frac{b - a}{m} \right) - 
                                  p \left(a + (k - 1) \frac{b - a}{m} \right) \right|
\end{align*}
By taking a limit for $m \to \infty$ we get:
\begin{align*} 
\lim_{m \to \infty} L(p, m) &= \lim_{m \to \infty} \sum_{k = 1}^{m} \frac{\left| p \left(a + k \frac{b - a}{m} \right) - p \left(a + (k - 1) \frac{b - a}{m} \right) \right|}{\frac{b - a}{m}} \cdot \frac{b - a}{m} \\
&= \lim_{m \to \infty} \sum_{k = 1}^{m} \left| \frac{p \left(t_{k-1} + \Delta t 
\right) - p \left( t_{k-1} \right)}{\Delta t} \right| \cdot \Delta t = \int_a^b{\left|\dot{p}(t)\right|  }{\dif t} = L(p)
\end{align*}

The discrete energy is given by:
\begin{align*} 
E(p, m) = \sum_{k = 1}^{m} \frac{\left( p \left(a + k \frac{b - a}{m} \right) - p \left(a + (k - 1) \frac{b - a}{m} \right) \right)^2}{\frac{b - a}{m}}
\end{align*}
By taking a limit for $m \to \infty$ we get:
\begin{align*} 
\lim_{m \to \infty} E(p, m) &= \lim_{m \to \infty} \sum_{k = 1}^{m} \left( \frac{ p \left(a + k \frac{b - a}{m} \right) - p \left(a + (k - 1) \frac{b - a}{m} \right)}{\frac{b - a}{m}} \right)^2 \cdot \frac{b - a}{m} \\
&= \lim_{m \to \infty} \sum_{k = 1}^{m} \left( \frac{ p \left(t_{k-1} + \Delta t 
\right) - p \left( t_{k-1} \right) }{\Delta t} \right)^2 \cdot \Delta t = \int_a^b{\left|\dot{p}(t)\right|^2  }{\dif t} = E(p)
\end{align*}

This way we see that the discrete length and energy are consistent with the length and energy defined at the beginning of this section.

Let's compute the length $L$ and energy $E$ of a curve $p(t)$ on a surface $S$  with parametrization ${\bf x}(u, v)$. 

The curve is given by $p(t) = {\bf x}(u(t), v(t)) : [a, b] \stackrel{c}{\rightarrow} \R^2 \stackrel{{\bf x}}{\rightarrow} S \subset \R^n$  (here $n \in \{2,3\}$). First we compute $\dot{p}(t)$:
\begin{align*} 
\dot{p}(t) = {\bf x}_u(u(t), v(t))\dot{u}(t)+{\bf x}_v(u(t), v(t))\dot{v}(t)
\end{align*}
The scalar product $\dot{p}(t) \cdot \dot{p}(t)$ with itself is:
\begin{align*} 
\dot{p}(t) \cdot \dot{p}(t) &= 
{\bf x}_u(u(t), v(t)) \cdot {\bf x}_u(u(t), v(t)) (\dot{u}(t))^2 \\
&+ 2 {\bf x}_u(u(t), v(t)) \cdot {\bf x}_v(u(t), v(t)) \dot{u}(t) \dot{v}(t) \\
&+ {\bf x}_v(u(t), v(t)) \cdot {\bf x}_v(u(t), v(t)) (\dot{v}(t))^2 
\end{align*}
Now we see the connection with the metric tensor (line element) of $S$:
\begin{align*} 
\dot{p}(t) \cdot \dot{p}(t) = 
g_{11}(u(t), v(t)) (\dot{u}(t))^2 + 2 g_{12}(u(t), v(t))\dot{u}(t)\dot{v}(t) + g_{22}(u(t), v(t))(\dot{v}(t))^2
\end{align*}
With $c(t)=\begin{pmatrix}u(t), v(t)\end{pmatrix}^t$ and the matrix $G(c(t))$ of the metric tensor, we can write:
\begin{align*} 
\dot{p}(t) \cdot \dot{p}(t) = 
\dot{c}(t)^t G(c(t)) \dot{c}(t)
=
\begin{pmatrix}
\dot{u}(t), \dot{v}(t)
\end{pmatrix}
\begin{pmatrix}
g_{11}(u(t), v(t)) & g_{12}(u(t), v(t)) \\
g_{21}(u(t), v(t)) & g_{22}(u(t), v(t))
\end{pmatrix}
\begin{pmatrix}
\dot{u}(t) \\ \dot{v}(t)
\end{pmatrix}
\end{align*}
Now we can define the length and energy of a curve on a surface:
\begin{definition}[Length of a curve on a surface]
The length of a curve $p(t)={\bf x}(c(t))$ on the surface ${\bf x}(u, v)$ is given by:
\begin{align*} 
L({\bf x}(c(t)))=\int_a^b\sqrt{\dot{c}(t)^t G(c(t)) \dot{c}(t)}{\dif t}
\end{align*}
\end{definition}
\begin{definition}[Energy of a curve on a surface]
The energy of a curve $p(t)={\bf x}(c(t))$ on the surface ${\bf x}(u, v)$ is given by:
\begin{align*} 
E({\bf x}(c(t)))=\int_a^b\dot{c}(t)^t G(c(t)) \dot{c}(t) {\dif t}
\end{align*}
\end{definition}

\subsection{Main theorem}

Here we prove two theorems which together give the main theorem \ref{main_thm} of this work:
\begin{theorem}\label{main_thm_1}
If a surface ${\bf x}(u, v)$ has the following orthogonal {\sc Liouville} line element: 
\begin{align*}
{\dif s}^2 = ( U_1(u) + V_1(v) ) {\dif u}^2 + ( U_2(u) + V_2(v) ) {\dif v}^2 
\end{align*}
then the diagonals in each rectangle formed by parameter lines on the surface have the same energy.
\end{theorem}
\begin{proof} We proceed in several steps:
\begin{enumerate}
\item We construct a rectangle $ABCD$ in the parameter domain $\left]u_{min}, u_{max}\right[ \times \left]v_{min}, v_{max}\right[$ of the surface ${\bf x}(u, v)$. First we take a point $M = (u_0, v_0)^t$ as the center of the rectangle. Second, we use a vector $\delta = (\alpha, \beta)^t$ (with $\alpha \leq 0$, $\beta \leq 0$) and set $A = M + \delta$. Then the first diagonal $CA$ of rectangle $ABCD$ has the parametrization $d_1(t) = M + t \delta$ with $t \in [-1, 1]$: for $t=-1$ we get the point $C$, for $t=0$ the point $M$ and for $t=1$ the point $A$. 
For the second diagonal we use the vector $\cj{\delta} = (-\alpha, \beta)^t$. The second diagonal $DB$ of rectangle $ABCD$ has the parametrization $d_2(t) = M + t \cj{\delta}$ with $t \in [-1, 1]$: for $t=-1$ we get the point $D$, for $t=0$ the point $M$ and for $t=1$ the point $B$. 
In choosing $M$ and $\delta$, we must be careful to ensure that $M$, $A$ and $C$ stay inside the definition domain $\left]u_{min}, u_{max}\right[ \times \left]v_{min}, v_{max}\right[$.
Then the sides of this rectangle $ABCD$ are parameter lines in the definition domain.

\item For the diagonals on the surface we have:
\begin{align*} 
q_1(t) &= 
\dot{d_1}(t)^t G(d_1(t)) \dot{d_1}(t) \\
&= 
\begin{pmatrix}
\alpha, \beta
\end{pmatrix}
\begin{pmatrix}
U_1(u_0 + t \alpha) + V_1(v_0 + t \beta) & 0 \\
0 &  U_2(u_0 + t \alpha) + V_2(v_0 + t \beta)
\end{pmatrix}
\begin{pmatrix}
\alpha \\ \beta
\end{pmatrix}\\
&= \alpha^2 (U_1(u_0 + t \alpha) + V_1(v_0 + t \beta)) + \beta^2 (U_2(u_0 + t \alpha) + V_2(v_0 + t \beta))
\end{align*}
and
\begin{align*} 
q_2(t) &= 
\dot{d_2}(t)^t G(d_2(t)) \dot{d_2}(t) \\
&= 
\begin{pmatrix}
-\alpha, \beta
\end{pmatrix}
\begin{pmatrix}
U_1(u_0 - t \alpha) + V_1(v_0 + t \beta) & 0 \\
0 &  U_2(u_0 - t \alpha) + V_2(v_0 + t \beta)
\end{pmatrix}
\begin{pmatrix}
-\alpha \\ \beta
\end{pmatrix}\\
&= \alpha^2 (U_1(u_0 - t \alpha) + V_1(v_0 + t \beta)) + \beta^2 (U_2(u_0 - t \alpha) + V_2(v_0 + t \beta))
\end{align*}

\item We show that the function $f(t) = q_1(t) - q_2(t)$ is an odd function of $t$, that means $f(-t) = -f(t)$:
\begin{align*} 
f(-t) =& q_1(-t) - q_2(-t) \\ 
      =& \alpha^2 (U_1(u_0 - t \alpha) + V_1(v_0 - t \beta)) + \beta^2 (U_2(u_0 - t \alpha) + V_2(v_0 - t \beta)) \\
       &- \alpha^2 (U_1(u_0 + t \alpha) + V_1(v_0 - t \beta)) - \beta^2 (U_2(u_0 + t \alpha) + V_2(v_0 - t \beta)) \\
      =& \alpha^2 U_1(u_0 - t \alpha) + \beta^2 U_2(u_0 - t \alpha) 
       - \alpha^2 U_1(u_0 + t \alpha) - \beta^2 U_2(u_0 + t \alpha) \\      
      =& \alpha^2 (U_1(u_0 - t \alpha) + V_1(v_0 + t \beta)) + \beta^2 (U_2(u_0 - t \alpha) + V_2(v_0 + t \beta)) \\
       &- \alpha^2 (U_1(u_0 + t \alpha) + V_1(v_0 + t \beta)) + \beta^2 (U_2(u_0 + t \alpha) + V_2(v_0 + t \beta)) \\
      =& q_2(t) - q_1(t) = -f(t)
\end{align*}

\item For the energies of the diagonals ${\bf x}(d_1(t))$ and ${\bf x}(d_2(t))$ on the surface we have:
\begin{align*} 
E({\bf x}(d_1(t))) - E({\bf x}(d_2(t))) &= \int_{-1}^{+1} q_1(t) {\dif t} - \int_{-1}^{+1} q_2(t) {\dif t} \\
&= \int_{-1}^{+1} q_1(t) - q_2(t) {\dif t} = \int_{-1}^{+1} f(t) {\dif t} = 0
\end{align*}
\end{enumerate}
We see that $E({\bf x}(d_1(t))) = E({\bf x}(d_2(t)))$ and we are done.
\end{proof}

\begin{theorem}\label{main_thm_2}
If the diagonals in each rectangle formed by parameter lines on a surface ${\bf x}(u, v)$ have the same energy, then the surface has the following orthogonal {\sc Liouville} line element: 
\begin{align*}
{\dif s}^2 = ( U_1(u) + V_1(v) ) {\dif u}^2 + ( U_2(u) + V_2(v) ) {\dif v}^2 
\end{align*}
\end{theorem}
\begin{proof}
We need several steps:
\begin{enumerate}
\item We construct a square $A_1 B_1 C_1 D_1$ in the parameter domain $\left]u_{min}, u_{max}\right[ \times \left]v_{min}, v_{max}\right[$ of the surface ${\bf x}(u, v)$. First we take the point $M = (u_0, v_0)^t$ as the center of the square. Second, we use a vector $\delta = (\alpha, \alpha)^t$ (with $\alpha < 0$) and set $A_1 = M + \delta$. Then the first diagonal $C_1 A_1$ of the square $A_1 B_1 C_1 D_1$ has the parametrization $d_1(t) = M + t \delta$ with $t \in [-1, 1]$. 
For the second diagonal we use the vector $\cj{\delta} = (-\alpha, \alpha)^t$. The second diagonal $D_1 B_1$ of square $A_1 B_1 C_1 D_1$ has the parametrization $d_2(t) = M + t \cj{\delta}$ with $t \in [-1, 1]$. 
In choosing $M$ and $\delta$, we must be careful to ensure that $M$, $A_1$ and $C_1$ stay inside the definition domain $\left]u_{min}, u_{max}\right[ \times \left]v_{min}, v_{max}\right[$.
Then the sides of this square $A_1 B_1 C_1 D_1$ are parameter lines in the definition domain.

\item The metric tensor of the surface has the general form:
\begin{align*}
G(u,v) =
\begin{pmatrix}
g_{11}(u,v) & g_{12}(u,v) \\
g_{12}(u,v) & g_{22}(u,v) 
\end{pmatrix}
\end{align*}
and we compute $q_1(t)$ and $q_2(t)$:
\begin{align*} 
q_1(t) &= 
\dot{d_1}(t)^t G(d_1(t)) \dot{d_1}(t) \\
&= 
\begin{pmatrix}
\alpha, \alpha
\end{pmatrix}
\begin{pmatrix}
g_{11}(u_0 + t \alpha, v_0 + t \alpha) & g_{12}(u_0 + t \alpha, v_0 + t \alpha) \\
g_{12}(u_0 + t \alpha, v_0 + t \alpha) & g_{22}(u_0 + t \alpha, v_0 + t \alpha)
\end{pmatrix}
\begin{pmatrix}
\alpha \\ \alpha
\end{pmatrix}\\
&= \alpha^2 (g_{11}(u_0 + t \alpha, v_0 + t \alpha) + 2 g_{12}(u_0 + t \alpha, v_0 + t \alpha) + g_{22}(u_0 + t \alpha, v_0 + t \alpha))
\end{align*}
and
\begin{align*} 
q_2(t) &= 
\dot{d_2}(t)^t G(d_2(t)) \dot{d_2}(t) \\
&= 
\begin{pmatrix}
-\alpha, \alpha
\end{pmatrix}
\begin{pmatrix}
g_{11}(u_0 - t \alpha, v_0 + t \alpha) & g_{12}(u_0 - t \alpha, v_0 + t \alpha) \\
g_{12}(u_0 - t \alpha, v_0 + t \alpha) & g_{22}(u_0 - t \alpha, v_0 + t \alpha)
\end{pmatrix}
\begin{pmatrix}
-\alpha \\ \alpha
\end{pmatrix}\\
&= \alpha^2 (g_{11}(u_0 - t \alpha, v_0 + t \alpha) - 2 g_{12}(u_0 - t \alpha, v_0 + t \alpha) + g_{22}(u_0 - t \alpha, v_0 + t \alpha))
\end{align*}
\item We know that the energies of the diagonals are equal and this implies that the function $f_1(t) = q_1(t) - q_2(t)$ is odd. Then the expression $f_1(-t) + f_1(t) = 0$ is the constant zero function.
For the particular value $t = 0$, $f_1(0) = q_1(0) - q_2(0) = 0$.
\begin{align*}
0 = f_1(0) = &\alpha^2 (g_{11}(u_0, v_0) + 2 g_{12}(u_0, v_0) + g_{22}(u_0, v_0)) \\
      &-\alpha^2 (g_{11}(u_0, v_0) - 2 g_{12}(u_0, v_0) + g_{22}(u_0, v_0)) = 4 \alpha^2 g_{12}(u_0, v_0)
\end{align*}
Therefore $g_{12}(u_0, v_0) = 0$ on the whole domain of definition, because $\alpha \neq 0$ and we can choose the point $(u_0, v_0)^t$ freely on the domain of definition. In what follows, we will use $g_{12}(u, v) = 0$ (the line element is orthogonal).
\item For $f_1(t)$ we now have (with $g_{12}(u, v) = 0$):
\begin{align*}
f_1(t) = &\alpha^2 (g_{11}(u_0 + t \alpha, v_0 + t \alpha) + g_{22}(u_0 + t \alpha, v_0 + t \alpha)\\
      &- g_{11}(u_0 - t \alpha, v_0 + t \alpha) - g_{22}(u_0 - t \alpha, v_0 + t \alpha)) 
\end{align*}
And we know that $c_1(t) := f_1(-t) + f_1(t) = 0$ is the constant zero function:
\begin{align*}
f_1(-t) + f_1(t) = &\alpha^2 (g_{11}(u_0 - t \alpha, v_0 - t \alpha) + g_{22}(u_0 - t \alpha, v_0 - t \alpha)\\
      &- g_{11}(u_0 + t \alpha, v_0 - t \alpha) - g_{22}(u_0 + t \alpha, v_0 - t \alpha)) \\
      &+ \alpha^2 (g_{11}(u_0 + t \alpha, v_0 + t \alpha) + g_{22}(u_0 + t \alpha, v_0 + t \alpha)\\
      &- g_{11}(u_0 - t \alpha, v_0 + t \alpha) - g_{22}(u_0 - t \alpha, v_0 + t \alpha)) 
\end{align*}
Therefore also the derivative $\frac{{\dif c_1}(t)}{\dif t} = 0$ with respect to $t$ is the constant zero function.
We can take another derivative $\frac{\dif^2 c_1(t)}{\dif t^2} = 0$ and know that it still is the constant zero function for all values of $t$. For $t=0$ we get 
\begin{align*}
0 = \frac{\dif^2 c_1(0)}{\dif t^2} 
  = 8 \alpha ^4 
  \left(
  \frac{\partial ^2g_{11}\left(u_0,v_0\right)}{\partial u\, \partial v}+
  \frac{\partial ^2g_{22}\left(u_0,v_0\right)}{\partial u\, \partial v}
  \right)
\end{align*}
Because $\alpha \neq 0$ we have the following relation between $g_{11}(u, v)$ and $g_{22}(u, v)$: 
\begin{align}
 \frac{\partial ^2g_{22}\left(u, v\right)}{\partial u\, \partial v} = 
-\frac{\partial ^2g_{11}\left(u, v\right)}{\partial u\, \partial v} \label{g22_from_g11}
\end{align}

\item We construct a rectangle $A_2 B_2 C_2 D_2$ in the parameter domain $\left]u_{min}, u_{max}\right[ \times \left]v_{min}, v_{max}\right[$ of the surface ${\bf x}(u, v)$. First we take the point $M = (u_0, v_0)^t$ as the center of the rectangle. Second, we use a vector $\delta = (\alpha, \alpha \varepsilon)^t$ (with $\alpha < 0$, $0 < \varepsilon < 1$) and set $A_2 = M + \delta$. Then the first diagonal $C_2 A_2$ of rectangle $A_2 B_2 C_2 D_2$ has the parametrization $d_3(t) = M + t \delta$ with $t \in [-1, 1]$. 
For the second diagonal we use the vector $\cj{\delta} = (-\alpha, \alpha \varepsilon)^t$. The second diagonal $D_2 B_2$ of rectangle $A_2 B_2 C_2 D_2$ has the parametrization $d_4(t) = M + t \cj{\delta}$ with $t \in [-1, 1]$. 
In choosing $M$ and $\delta$, we must be careful to ensure that $M$, $A_2$ and $C_2$ stay inside the definition domain $\left]u_{min}, u_{max}\right[ \times \left]v_{min}, v_{max}\right[$.
Then the sides of this rectangle $A_2 B_2 C_2 D_2$ are parameter lines in the definition domain.

\item The metric tensor of the surface has the form (remember that $g_{12}(u, v) = 0$):
\begin{align*}
G(u,v) =
\begin{pmatrix}
g_{11}(u,v) & 0 \\
          0 & g_{22}(u,v) 
\end{pmatrix}
\end{align*}
and we compute $q_3(t)$ and $q_4(t)$:
\begin{align*} 
q_3(t) &= 
\dot{d_3}(t)^t G(d_3(t)) \dot{d_3}(t) \\
&= 
\begin{pmatrix}
\alpha, \alpha \varepsilon
\end{pmatrix}
\begin{pmatrix}
g_{11}(u_0 + t \alpha, v_0 + t \alpha \varepsilon) & 0 \\
                                     0 & g_{22}(u_0 + t \alpha, v_0 + t \alpha \varepsilon)
\end{pmatrix}
\begin{pmatrix}
\alpha \\ \alpha \varepsilon
\end{pmatrix}\\
&= \alpha^2 (g_{11}(u_0 + t \alpha, v_0 + t \alpha \varepsilon) + \varepsilon^2 g_{22}(u_0 + t \alpha, v_0 + t \alpha \varepsilon))
\end{align*}
and
\begin{align*} 
q_4(t) &= 
\dot{d_4}(t)^t G(d_4(t)) \dot{d_4}(t) \\
&= 
\begin{pmatrix}
-\alpha, \alpha \varepsilon
\end{pmatrix}
\begin{pmatrix}
g_{11}(u_0 - t \alpha, v_0 + t \alpha \varepsilon) & 0 \\
                                     0 & g_{22}(u_0 - t \alpha, v_0 + t \alpha \varepsilon)
\end{pmatrix}
\begin{pmatrix}
-\alpha \\ \alpha \varepsilon
\end{pmatrix}\\
&= \alpha^2 (g_{11}(u_0 - t \alpha, v_0 + t \alpha \varepsilon) + \varepsilon^2 g_{22}(u_0 - t \alpha, v_0 + t \alpha \varepsilon))
\end{align*}

\item We know that the energies of the diagonals are equal and this implies that the function $f_2(t) = q_3(t) - q_4(t)$ is odd. Then the expression $f_2(-t) + f_2(t) = 0$ is the constant zero function.

\item For $f_2(t)$ we have:
\begin{align*}
f_2(t) = &\alpha^2 (g_{11}(u_0 + t \alpha, v_0 + t \alpha \varepsilon) + \varepsilon^2 g_{22}(u_0 + t \alpha, v_0 + t \alpha \varepsilon)) \\
      &-\alpha^2 (g_{11}(u_0 - t \alpha, v_0 + t \alpha \varepsilon) + \varepsilon^2 g_{22}(u_0 - t \alpha, v_0 + t \alpha \varepsilon))
\end{align*}
And we know that $c_2(t) := f_2(-t) + f_2(t) = 0$.
\begin{align*}
c_2(t) =& \alpha^2 (
  \varepsilon^2 g_{22}\left(u_0 - t \alpha, v_0 - t \alpha \varepsilon\right)-
  \varepsilon^2 g_{22}\left(u_0 - t \alpha, v_0 + t \alpha \varepsilon\right)\\
&-\varepsilon^2 g_{22}\left(u_0 + t \alpha, v_0 - t \alpha \varepsilon\right)+
  \varepsilon^2 g_{22}\left(u_0 + t \alpha, v_0 + t \alpha \varepsilon\right)\\
&+g_{11}\left(u_0 - t \alpha, v_0 - t \alpha \varepsilon\right)
 -g_{11}\left(u_0 - t \alpha, v_0 + t \alpha \varepsilon\right)\\
&-g_{11}\left(u_0 + t \alpha, v_0 - t \alpha \varepsilon\right)
 +g_{11}\left(u_0 + t \alpha, v_0 + t \alpha \varepsilon\right)
)
\end{align*}
Therefore also the derivative $\frac{{\dif c_2}(t)}{\dif t} = 0$ with respect to $t$ is the constant zero function.
We can take another derivative $\frac{\dif^2 c_2(t)}{\dif t^2} = 0$ and know that it still is the constant zero function for all values of $t$. We use relation (\ref{g22_from_g11}) and we get for $t = 0$: 
\begin{align*}
0 = \frac{\dif^2 c_2(0)}{\dif t^2} 
  = 8 \alpha^4 \varepsilon \left(1 - \varepsilon^2\right) \frac{\partial ^2g_{11}\left(u_0,v_0\right)}{\partial u\,
     \partial v}
\end{align*}
Because $\alpha \neq 0$ and $0 < \varepsilon < 1$ we have (combined with (\ref{g22_from_g11})): 
\begin{align*}
 \frac{\partial ^2g_{11}\left(u, v\right)}{\partial u\, \partial v} = 0 \quad \text{and} \quad
 \frac{\partial ^2g_{22}\left(u, v\right)}{\partial u\, \partial v} = 0
\end{align*}
\item By solving each of these differential equations separately with respect to $u$ and $v$ we get:
\begin{align*}
g_{11}(u, v) = U_1(u) + V_1(v) \quad \text{and} \quad g_{22}(u, v) = U_2(u) + V_2(v)  
\end{align*}
\end{enumerate}
And we are done.
\end{proof}

The previous two theorems combined give together the following main theorem of this work:
\begin{theorem}[Main theorem]\label{main_thm}
If and only if a surface ${\bf x}(u, v)$ has the following orthogonal {\sc Liouville} line element: 
\begin{align*}
{\dif s}^2 = ( U_1(u) + V_1(v) ) {\dif u}^2 + ( U_2(u) + V_2(v) ) {\dif v}^2 
\end{align*}
the diagonals in each rectangle formed by parameter lines on the surface have the same energy.
\end{theorem}
\begin{corollary}
It immediately follows that for all special cases of this line element the diagonals have the same energy.
\end{corollary}

\section{{\sc Liouville} surfaces}

In this section we define and name some special line elements or parametrizations of surfaces and give examples of {\sc Liouville} surfaces.

\subsection{Definitions}

Let a surface be given as a parametrization ${\bf x}(u, v) : \left]u_{min}, u_{max}\right[ \times \left]v_{min}, v_{max}\right[ \subset \R^2 \mapsto \R^n$ (in this work we have $n \in \{2,3\}$ but the theory is valid in any dimension).

\begin{definition}[Line element and first fundamental form]
We define the line element ${\dif s}$ of the surface ${\bf x}(u, v)$ by the first fundamental form $ds^2$: 
\begin{align}
{\dif s}^2 = g_{11}(u, v) {\dif u}^2 + 2 g_{12}(u, v) {\dif u} {\dif v} + g_{22}(u, v) {\dif v}^2
\end{align}
where 
$g_{11}(u, v) = {\bf x}_u(u, v) \boldsymbol{\cdot} {\bf x}_u(u, v)$, 
$g_{12}(u, v) = g_{21}(u, v) = {\bf x}_u(u, v) \boldsymbol{\cdot} {\bf x}_v(u, v)$
and
$g_{22}(u, v) = {\bf x}_v(u, v) \boldsymbol{\cdot} {\bf x}_v(u, v)$.
\end{definition}

When $g_{12} = 0$, the $u$-lines $v = \rm const.$ and $v$-lines $u = \rm const.$ are orthogonal to each other on the surface and we speak of an orthogonal parametrization:

\begin{definition}[Orthogonal parametrization]
When the line element ${\dif s}$ of the surface ${\bf x}(u, v)$ has the following form: 
\begin{align}
{\dif s}^2 = g_{11}(u, v) {\dif u}^2 + g_{22}(u, v) {\dif v}^2
\end{align}
with 
$g_{12}(u, v) = g_{21}(u, v) = 0$, the parametrization of the surface is called orthogonal.
\end{definition}

When $g_{11} = g_{22}$, $g_{12} = 0$ the parametrization is orthogonal and locally conformal (preserves angles between curves (their tangent vectors)):

\begin{definition}[Isothermal (or conformal) parametrization]
When the line element ${\dif s}$ of the surface ${\bf x}(u, v)$ has the following form: 
\begin{align}
{\dif s}^2 = g_{11}(u, v) ( {\dif u}^2 + {\dif v}^2 )
\end{align}
with 
$g_{11}(u, v) = g_{22}(u, v)$ and
$g_{12}(u, v) = g_{21}(u, v) = 0$, the parametrization of the surface is called isothermal.
\end{definition}

Now we can define {\sc Liouville} and {\sc Clairaut} parametrizations:

\begin{definition}[Orthogonal {\sc Liouville} parametrization]
When the line element ${\dif s}$ of the surface ${\bf x}(u, v)$ has the following form: 
\begin{align}
{\dif s}^2 = ( U_1(u) + V_1(v) ) {\dif u}^2 + ( U_2(u) + V_2(v) ) {\dif v}^2 
\end{align}
with 
$\frac{\partial^2}{\partial u \partial v} g_{11}(u, v) = \frac{\partial^2}{\partial u \partial v} g_{22}(u, v) = 0$ and
$g_{12}(u, v) = g_{21}(u, v) = 0$, the parametrization of the surface is called orthogonal {\sc Liouville} parametrization.
\end{definition}
A special case is (when $g_{11} = g_{22}$):

\begin{definition}[Isothermal (or classical) {\sc Liouville} parametrization]
When the line element ${\dif s}$ of the surface ${\bf x}(u, v)$ has the following form: 
\begin{align}
{\dif s}^2 = ( U(u) + V(v) ) ( {\dif u}^2 + {\dif v}^2 )
\end{align}
with 
$g_{11}(u, v) = g_{22}(u, v) = U(u) + V(v)$ and
$g_{12}(u, v) = g_{21}(u, v) = 0$, the parametrization of the surface is called isothermal {\sc Liouville} parametrization.
\end{definition}

The {\sc Liouville} parametrizations specialize further to {\sc Clairaut} parametrizations, when $g_{11}$ and $g_{22}$ depend only on $u$ (or on $v$):

\begin{definition}[Orthogonal {\sc Clairaut} parametrization in $u$]
When the line element ${\dif s}$ of the surface ${\bf x}(u, v)$ has the following form: 
\begin{align}
{\dif s}^2 = U_1(u) {\dif u}^2 + U_2(u) {\dif v}^2 
\end{align}
with 
$\frac{\partial}{\partial v} g_{11}(u, v) = \frac{\partial}{\partial v} g_{22}(u, v) = 0$ and
$g_{12}(u, v) = g_{21}(u, v) = 0$, the parametrization of the surface is called orthogonal {\sc Clairaut} parametrization in $u$. (The other case is that of the orthogonal {\sc Clairaut} parametrization in $v$.) 
\end{definition}

A special case is (when $g_{11} = g_{22} = U(u)$ or $g_{11} = g_{22} = V(v)$):

\begin{definition}[Isothermal {\sc Clairaut} parametrization in $u$]
When the line element ${\dif s}$ of the surface ${\bf x}(u, v)$ has the following form: 
\begin{align}
{\dif s}^2 = U(u) ( {\dif u}^2 + {\dif v}^2 )
\end{align}
with 
$g_{11}(u, v) = g_{22}(u, v) = U(u)$ and
$g_{12}(u, v) = g_{21}(u, v) = 0$, the parametrization of the surface is called isothermal {\sc Clairaut} parametrization in $u$. (The other case is that of the isothermal {\sc Clairaut} parametrization in $v$.)
\end{definition}

Other two special cases worth mentioning are:

\begin{definition}[Orthogonal {\sc Liouville} parametrization in $U_1(u)$ and $V_2(v)$]
When the line element ${\dif s}$ of the surface ${\bf x}(u, v)$ has the following form: 
\begin{align}
{\dif s}^2 = U_1(u) {\dif u}^2 + V_2(v) {\dif v}^2 
\end{align}
with 
$\frac{\partial}{\partial v} g_{11}(u, v) = \frac{\partial}{\partial u} g_{22}(u, v) = 0$ and
$g_{12}(u, v) = g_{21}(u, v) = 0$, the parametrization of the surface is called orthogonal {\sc Liouville} parametrization in $U_1(u)$ and $V_2(v)$.
(The other case is that of the orthogonal {\sc Liouville} parametrization in $V_1(v)$ and $U_2(u)$.) 
\end{definition}

The next line element is relevant for the {\sc Ivory}--property of the geodesic diagonals:
 
\begin{definition}[{\sc Stäckel} parametrization]
When the line element ${\dif s}$ of the surface ${\bf x}(u, v)$ has the following form: 
\begin{align}
{\dif s}^2 =
\begin{vmatrix}
U(u)   & V(v)   \\
U_1(u) & V_1(v)
\end{vmatrix}
\left( \frac{{\dif u}^2}{V_1(v)} - \frac{{\dif v}^2}{U_1(u)} \right)
\end{align}
with 
$g_{12}(u, v) = g_{21}(u, v) = 0$, the parametrization of the surface is called {\sc Stäckel} parametrization.
\end{definition}

\begin{remark}
The isothermal {\sc Liouville} parametrization is a {\sc Stäckel} parametrization because we can write it as follows:
\begin{align*}
{\dif s}^2 = 
\begin{vmatrix}
U(u) & V(v)   \\
-1   & 1
\end{vmatrix}
\left( {\dif u}^2 + {\dif v}^2 \right) =
\left( U(u) + V(v) \right) \left( {\dif u}^2 + {\dif v}^2 \right)
\end{align*}
\end{remark}

\begin{remark}
The isothermal {\sc Clairaut} parametrization in $u$ is a {\sc Stäckel} parametrization:
\begin{align*}
{\dif s}^2 = 
\begin{vmatrix}
U(u) & 0   \\
-1   & 1
\end{vmatrix}
\left( {\dif u}^2 + {\dif v}^2 \right) = U(u) \left( {\dif u}^2 + {\dif v}^2 \right)
\end{align*}
The isothermal {\sc Clairaut} parametrization in $v$ is a {\sc Stäckel} parametrization:
\begin{align*}
{\dif s}^2 = 
\begin{vmatrix}
 0   & V(v)   \\
-1   & 1
\end{vmatrix}
\left( {\dif u}^2 + {\dif v}^2 \right) = V(v) \left( {\dif u}^2 + {\dif v}^2 \right)
\end{align*}
\end{remark}

\begin{remark}
The orthogonal {\sc Clairaut} parametrization in $u$ is a {\sc Stäckel} parametrization:
\begin{align*}
{\dif s}^2 =
\begin{vmatrix}
U_1(u) & 0   \\
\frac{U_1(u)}{U_2(u)} & -1
\end{vmatrix}
\left( - {\dif u}^2 - \frac{U_2(u)}{U_1(u)} {\dif v}^2 \right) = U_1(u) {\dif u}^2 + U_2(u) {\dif v}^2 
\end{align*}
The orthogonal {\sc Clairaut} parametrization in $v$ is a {\sc Stäckel} parametrization:
\begin{align*}
{\dif s}^2 =
\begin{vmatrix}
0 & V_2(v)   \\
-1 &\frac{V_2(v)}{V_1(v)}
\end{vmatrix}
\left(\frac{V_1(v)}{V_2(v)}{\dif u}^2 + {\dif v}^2 \right) = V_1(v) {\dif u}^2 + V_2(v) {\dif v}^2 
\end{align*}
\end{remark}

\begin{remark}
The orthogonal {\sc Liouville} parametrization in $U_1(u)$ and $V_2(v)$ is a {\sc Stäckel} parametrization:
\begin{align*}
{\dif s}^2 =
\begin{vmatrix}
U_1(u) & 0   \\
U_1(u) & -V_2(v)
\end{vmatrix}
\left( -\frac{{\dif u}^2}{V_2(v)} - \frac{{\dif v}^2}{U_1(u)}  \right) = U_1(u) {\dif u}^2 + V_2(v) {\dif v}^2 
\end{align*}
The orthogonal {\sc Liouville} parametrization in $V_1(v)$ and $U_2(u)$ is in general not a {\sc Stäckel} parametrization. We can write:
\begin{align*}
{\dif s}^2 =
\begin{vmatrix}
U(u)   & V(v)   \\
-\frac{1}{U_2(u)} & \frac{1}{V_1(v)}
\end{vmatrix}
\left( V_1(v) {\dif u}^2 + U_2(u) {\dif v}^2 \right)
\end{align*}
and we see that this is only then an orthogonal {\sc Liouville} parametrization in $V_1(v)$ and $U_2(u)$ when the following determinant is a non-zero constant:
\begin{align*}
\begin{vmatrix}
U(u)   & V(v)   \\
-\frac{1}{U_2(u)} & \frac{1}{V_1(v)}
\end{vmatrix}
=\frac{U(u)}{V_1(v)} + \frac{V(v)}{U_2(u)}
=\text{const.}
\end{align*}

\end{remark}

\begin{remark}
The orthogonal {\sc Liouville} parametrization is in general not a {\sc Stäckel} parametrization. We only get  special cases when both $U_1(u)$ and $V_1(v)$ in the {\sc Stäckel} parametrization are non-zero constants.
\end{remark}

\subsection{Examples of {\sc Liouville} surfaces}

In this section we want to give examples of {\sc Liouville} surfaces. On all surfaces with a {\sc Liouville} line element (or a {\sc Clairaut} line element as special case) the diagonals of a parameter line rectangle have the same energy (cf. theorem \ref{main_thm_1}). 

\subsubsection{Surfaces of constant {\sc Gaussian} curvature} The first three examples are surfaces of constant {\sc Gaussian} curvature: plane, sphere, pseudosphere.

\begin{example}[Plane]
The plane admits four different (up to uniform scaling and {\sc Euclidean} motions - these are similarity transforms) isothermal {\sc Liouville} parametrizations (the first four examples below, see figure \ref{fig:plane_li_2}):
\begin{enumerate}
\item {\sc Cartesian} coordinates: 
\begin{align*}
{\bf x}(u, v) = \begin{pmatrix}u\\v\end{pmatrix}
\end{align*}
The line element is isothermal {\sc Clairaut} in $u$ (or $v$): ${\dif s}^2 = {\dif u}^2 + {\dif v}^2$.

\item Polar coordinates: 
\begin{align*}
{\bf x}(u, v) = \begin{pmatrix}e^u \cos(v)\\e^u \sin(v)\end{pmatrix}
\end{align*}
The line element is isothermal {\sc Clairaut} in $u$: ${\dif s}^2 = e^{2 u}({\dif u}^2 + {\dif v}^2)$.

\item Parabolic coordinates: 
\begin{align*}
{\bf x}(u, v) = \begin{pmatrix}u^2 - v^2\\2 u v\end{pmatrix}
\end{align*}
The line element is isothermal {\sc Liouville}: ${\dif s}^2 = 4 (u^2 + v^2) ({\dif u}^2 + {\dif v}^2)$.

\item Elliptic coordinates: 
\begin{align*}
{\bf x}(u, v) = \begin{pmatrix}\cos(u)\cosh(v)\\-\sin(u)\sinh(v)\end{pmatrix}
\end{align*}
They are isothermal {\sc Liouville}: ${\dif s}^2 = \frac{1}{2} (\cosh (2 v) - \cos (2 u)) ({\dif u}^2 + {\dif v}^2)$.

\item Standard polar coordinates: 
\begin{align*}
{\bf x}(u, v) = \begin{pmatrix}u \cos(v)\\u \sin(v)\end{pmatrix}
\end{align*}
They are orthogonal {\sc Clairaut} in $u$: ${\dif s}^2 = {\dif u}^2 + u^2 {\dif v}^2$. These are also called geodesic parallel coordinates (because $g_{11}=1$ and $g_{12}=0$).
\item Orthogonal {\sc Liouville} coordinates: 
\begin{align*}
{\bf x}(u, v) = \begin{pmatrix}v^5\\u^2\end{pmatrix}
\end{align*}
They are orthogonal {\sc Liouville} coordinates in $U_1(u)$ and $V_2(v)$: ${\dif s}^2 = 4 u^2 {\dif u}^2 + 25 v^8 {\dif v}^2$. See figure \ref{fig:plane_li_5}.
\end{enumerate}
\end{example}

\begin{example}[Unit sphere centered at origin]
The sphere admits several {\sc Liouville} parametrizations:
\begin{enumerate}
\item Standard parametrization as surface of rotation:
\begin{align*}
{\bf x}(u, v) = \begin{pmatrix}
\cos(u) \cos(v)\\ 
\cos(u) \sin(v)\\
\sin(u)\end{pmatrix}
\end{align*}
The line element is orthogonal {\sc Clairaut} in $u$: ${\dif s}^2 = {\dif u}^2 + \cos^2(u){\dif v}^2$.

\item {\sc Mercator} coordinates: 
\begin{align*}
{\bf x}(u, v) = \begin{pmatrix}
\frac{\cos(v)}{\cosh(u)}\\ 
\frac{\sin(v)}{\cosh(u)}\\ 
\tanh(u)
\end{pmatrix}
\end{align*}
The line element is isothermal {\sc Clairaut} in $u$: ${\dif s}^2 = \frac{1}{\cosh^2(u)}({\dif u}^2 + {\dif v}^2)$. See left image in figure \ref{fig:sphere_li}.

\item Elliptic coordinates (see article \cite{boy}, formula (3.51)) with {\sc Jacobi} elliptic functions (where $i=\sqrt{-1}$ and the modulus is $m=k^2=\frac{1}{2}$):
\begin{align*}
{\bf x}(u, v) = \begin{pmatrix}
\frac{1}{\sqrt{2}} \sn\left(i u\right) \sn\left(v\right)\\
                 i \cn\left(i u\right) \cn\left(v\right)\\
          \sqrt{2} \dn\left(i u\right) \dn\left(v\right)
\end{pmatrix}
\end{align*}
This is isothermal {\sc Liouville}: ${\dif s}^2 = \frac{1}{2} \left(\sn^2\left(i u\right) - \sn^2\left(  v\right)\right)({\dif u}^2 + {\dif v}^2)$. The parameter lines form two families of confocal geodesic ellipses. See right image in figure \ref{fig:sphere_li}.
\end{enumerate}
\begin{figure}[hbt]
\centering
$\begin{array}{cccc}
\includegraphics[scale=0.5]{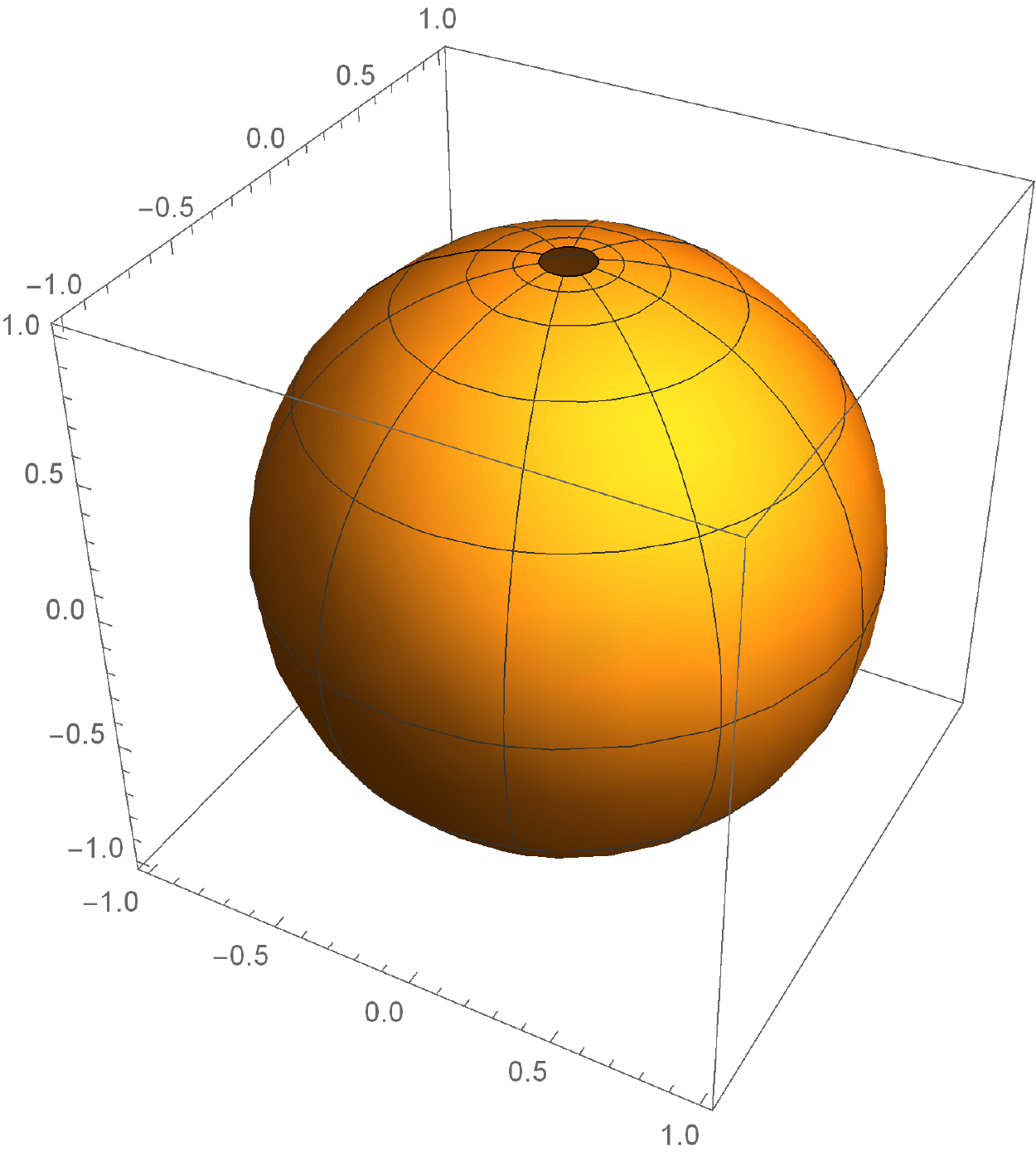} &
\includegraphics[scale=0.5]{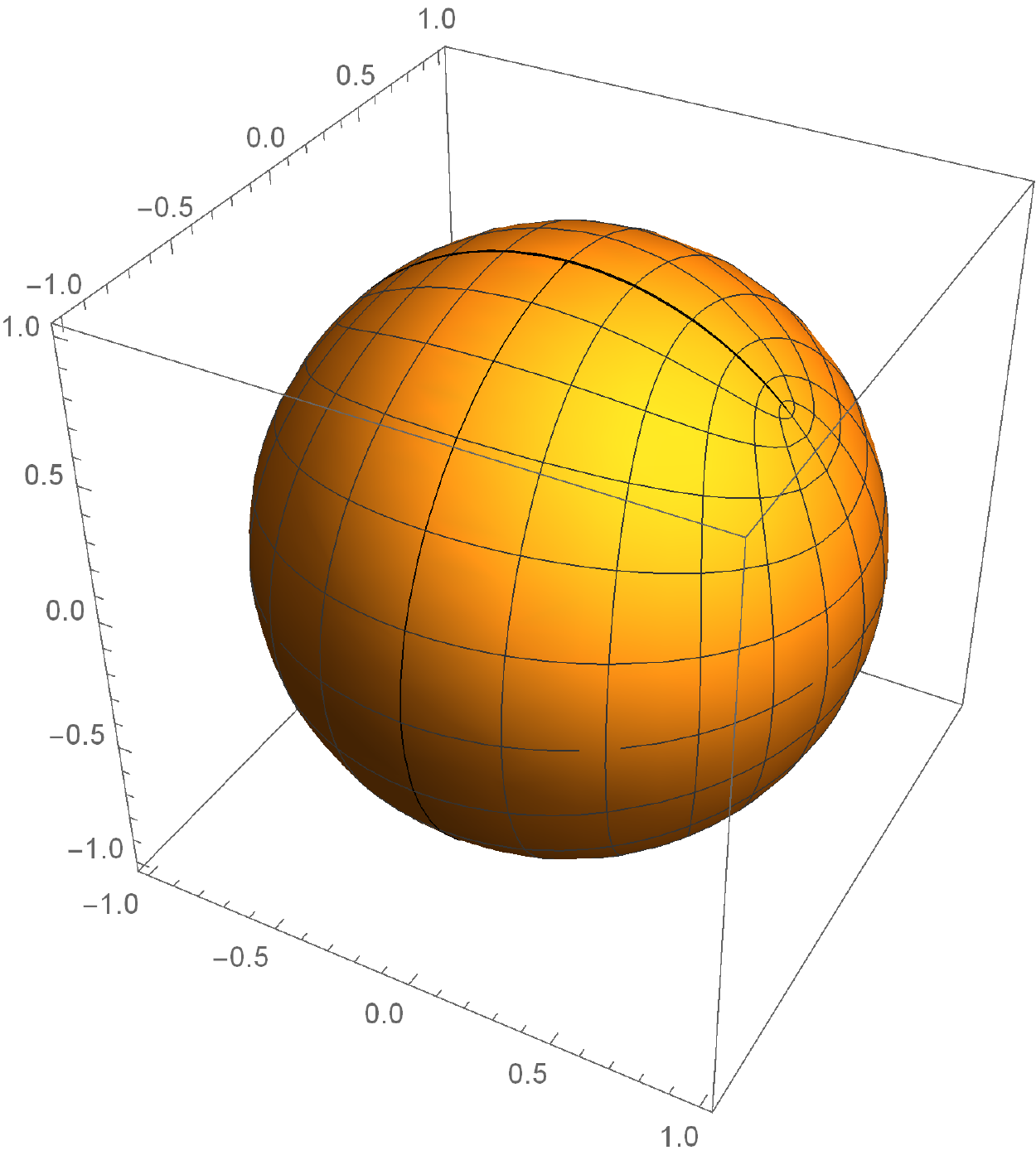}
\end{array}$
\caption{Isothermal {\sc Liouville} coordinates on the sphere}
\label{fig:sphere_li}
\end{figure}
\end{example}

\begin{example}[Pseudosphere]
The pseudosphere admits the following {\sc Liouville} parametrizations:
\begin{enumerate}
\item Standard parametrization as surface of rotation: 
\begin{align*}
{\bf x}(u, v) = \begin{pmatrix}
\frac{\cos(v)}{\cosh(u)}\\ 
\frac{\sin(v)}{\cosh(u)}\\ 
u-\frac{\sinh(u)}{\cosh(u)}
\end{pmatrix}
\end{align*}
The line element is orthogonal {\sc Clairaut} in $u$: ${\dif s}^2 = \frac{\sinh(u)}{\cosh(u)}{\dif u}^2 + \frac{1}{\cosh(u)}{\dif v}^2$. See left image in figure \ref{fig:pseudosphere}.

\item {\sc Liouville} coordinates: 
\begin{align*}
{\bf x}(u, v) = \begin{pmatrix}
\frac{\cos (v)}{u} \\
\frac{\sin (v)}{u} \\
\arccosh(u) - \frac{\sqrt{u^2 - 1}}{u}
\end{pmatrix}
\end{align*}
The line element is isothermal {\sc Clairaut} in $u$: ${\dif s}^2 = \frac{1}{u^2}({\dif u}^2 + {\dif v}^2)$. See right image in figure \ref{fig:pseudosphere}.
\end{enumerate}
\begin{figure}[hbt]
\centering
$\begin{array}{cccc}
\includegraphics[scale=0.5]{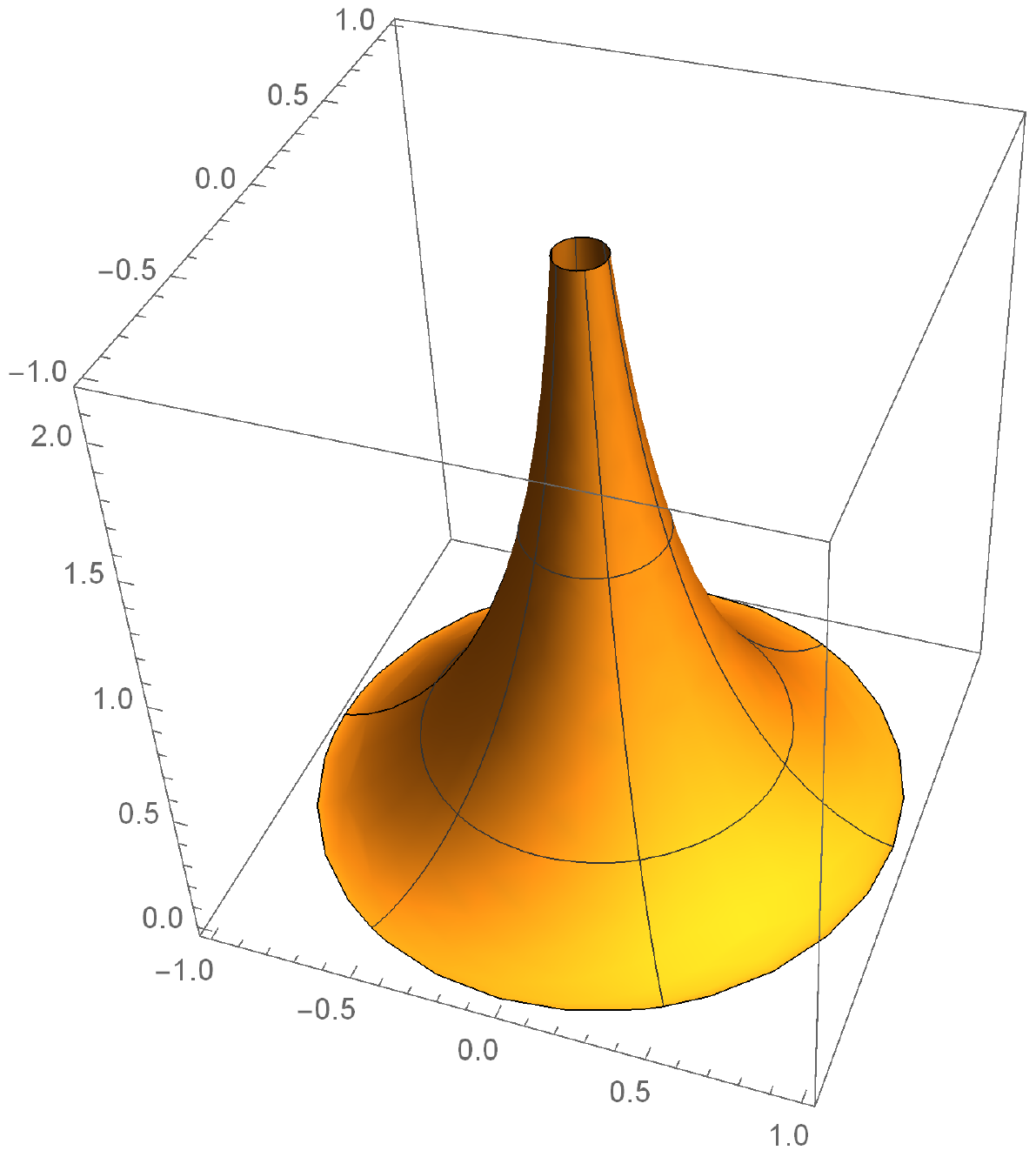} &
\includegraphics[scale=0.5]{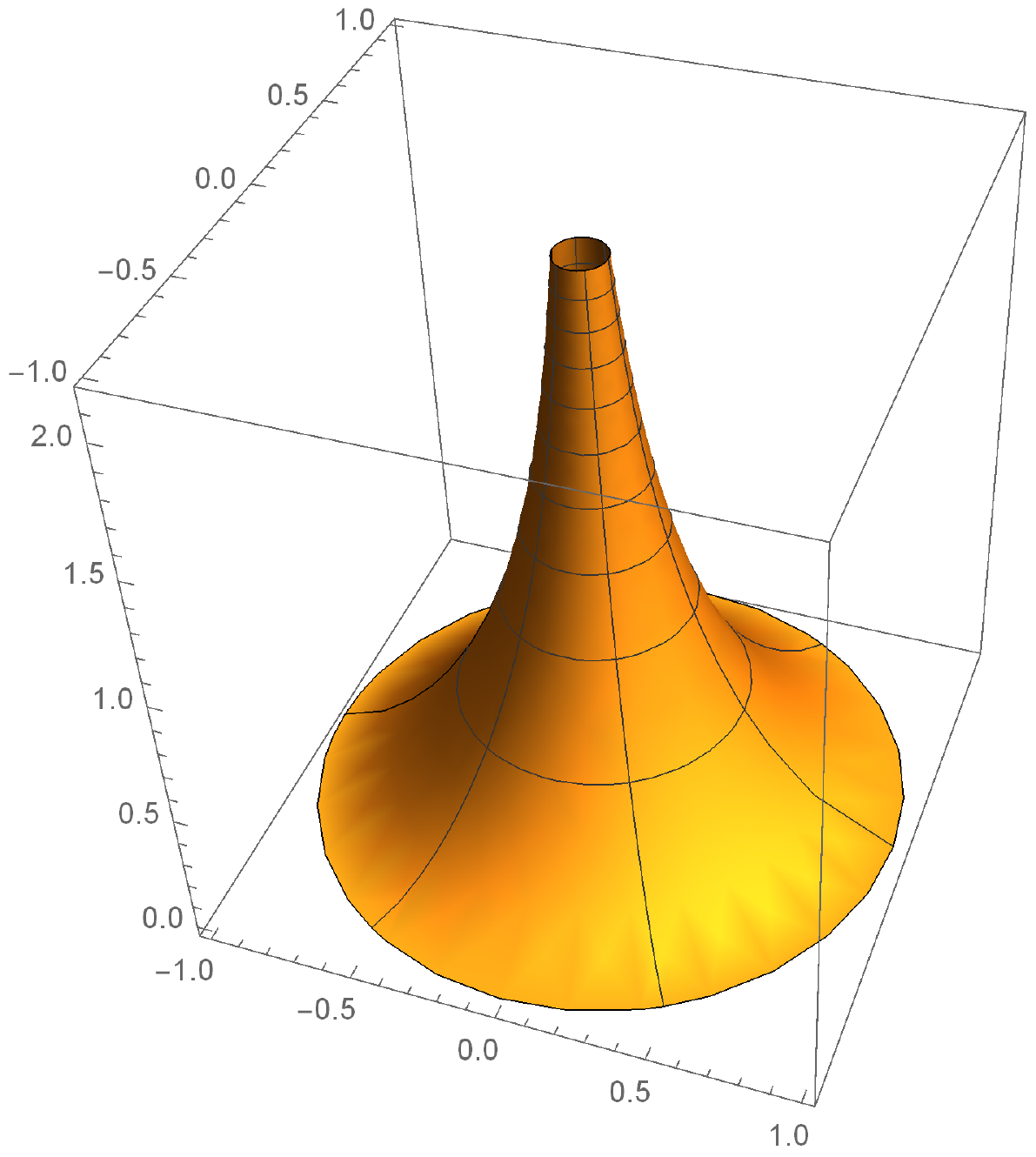}
\end{array}$
\caption{Pseudosphere; in the left image orthogonal {\sc Clairaut} parametrization and in the right image with isothermal {\sc Clairaut} parametrization}
\label{fig:pseudosphere}
\end{figure}
\end{example}

\subsubsection{Surfaces of rotation}
\begin{example}[Surface of rotation]
The surface of rotation admits the following {\sc Liouville} parametrizations:
\begin{enumerate}
\item Standard parametrization as surface of rotation:
\begin{align*}
{\bf x}(u, v) = \begin{pmatrix}
r(u)\cos(v)\\ 
r(u)\sin(v)\\ 
h(u)
\end{pmatrix}
\end{align*}
This is orthogonal {\sc Clairaut} in $u$: ${\dif s}^2 = ((h'(u))^2+(r'(u))^2){\dif u}^2 + (r(u))^2 {\dif v}^2$.

\item {\sc Liouville} coordinates: Let's start with the previous parametrization and express $u$ as a function of $t$. We want to achieve an isothermal {\sc Clairaut} parametrization in $t$, that means:
$((h'(u))^2+(r'(u))^2)\frac{{\dif u}^2}{{\dif t}^2} = (r(u))^2$. Then we get:
\begin{align*}
\int {\dif t} &= \int \frac{\sqrt{((h'(u))^2+(r'(u))^2)}}{(r(u))^2}{\dif u} \\
t &= F(u) \\
F^{-1}(t) &= u
\end{align*}
The new parametrization is then:
\begin{align*}
{\bf x}(F^{-1}(t), v) = \begin{pmatrix}
r(F^{-1}(t))\cos(v)\\ 
r(F^{-1}(t))\sin(v)\\ 
h(F^{-1}(t))
\end{pmatrix}
\end{align*}
This is isothermal {\sc Clairaut} in $t$: ${\dif s}^2 = (r(F^{-1}(t)))^2({\dif t}^2 + {\dif v}^2)$.
\end{enumerate}
\end{example}

\subsubsection{Surfaces of translation (Schiebflächen)}
\begin{example}[Parabolic cylinder]
The following parametrization: 
\begin{align*}
{\bf x}(u, v) = \begin{pmatrix}
u\\ 
u^2+v^2\\ 
u^2-v^2
\end{pmatrix}
\end{align*}
is a surface of translation with implicit equation: $2 x^2 = y + z$ and has as line element: ${\dif s}^2 = (1 + 8 u^2){\dif u}^2 + 8 v^2 {\dif v}^2$.
This is orthogonal {\sc Liouville} in $U_1(u)$ and $V_2(v)$.
\end{example}

\begin{example}[Parabolic cylinder]
The following parametrization: 
\begin{align*}
{\bf x}(u, v) = \begin{pmatrix}
u\\ 
v\\ 
u^2
\end{pmatrix}
\end{align*}
is a surface of translation with implicit equation: $z = x^2$ and has as line element: ${\dif s}^2 = (1 + 4 u^2){\dif u}^2 + {\dif v}^2$.
This is orthogonal {\sc Clairaut} in $u$.
\end{example}

\begin{example}[Plane]
The following parametrization: 
\begin{align*}
{\bf x}(u, v) = \begin{pmatrix}
u\\ 
u+v\\ 
u-v
\end{pmatrix}
\end{align*}
is a plane as surface of translation with implicit equation: $2x = y + z$ and has as line element: ${\dif s}^2 = 3{\dif u}^2 + 2{\dif v}^2$.
This is orthogonal {\sc Clairaut} in $u$ (or $v$).
\end{example}

\subsubsection{Minimal {\sc Liouville} surfaces (from \cite{bbw}}
\begin{example}[{\sc Enneper} surface]
The following polynomial parametrization of the {\sc Enneper} minimal surface: 
\begin{align*}
{\bf x}(u, v) = \begin{pmatrix}
v \left(-3 u^2  +v^2+3\right)\\
u \left(   u^2-3 v^2+3\right)\\
6 u v
\end{pmatrix}
\end{align*}
has the line element: ${\dif s}^2 = 9 ( 1 + u^2 + v^2 )^2 ({\dif u}^2 + {\dif v}^2)$.
This is only isothermal but not a {\sc Liouville} line element. This parametrization is therefore a counterexample to the {\sc Liouville} parametrizations and the diagonals do not have the same energy in this parametrization. See left image in figure \ref{fig:minimal_enneper}.
\begin{figure}[hbt]
\centering
$\begin{array}{cccc}
\includegraphics[scale=0.5]{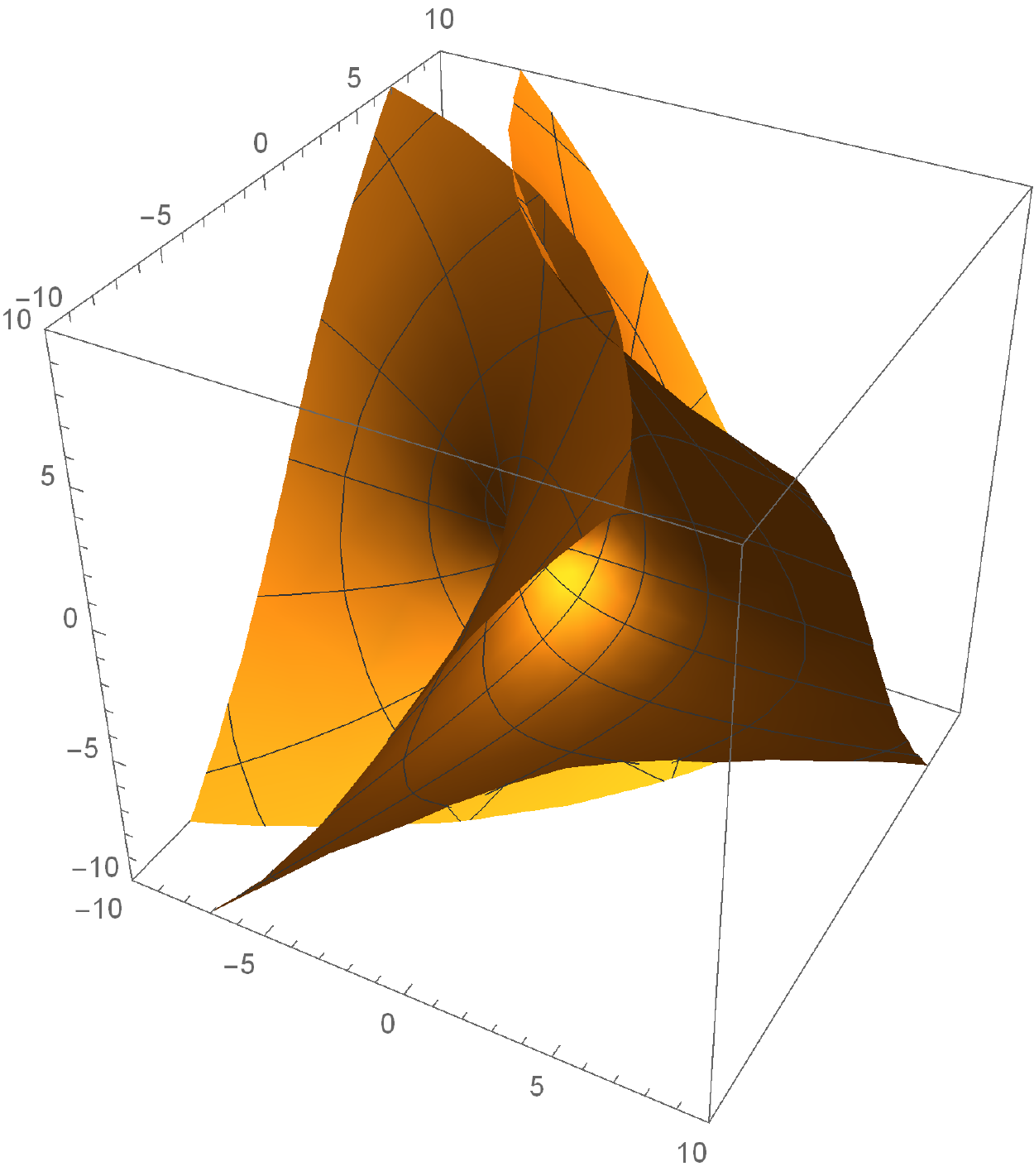} &
\includegraphics[scale=0.5]{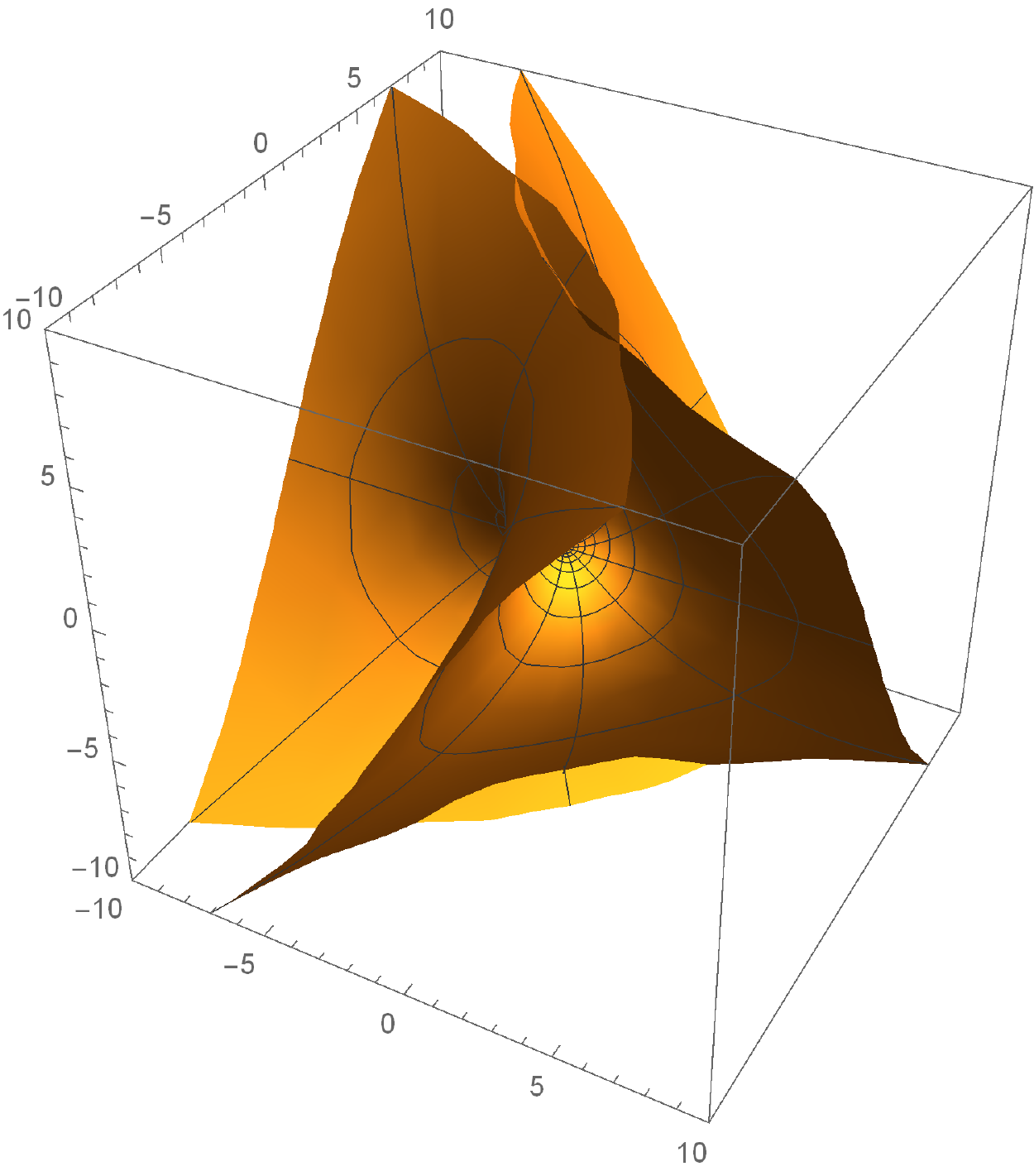}
\end{array}$
\caption{{\sc Enneper} surface; left image with isothermal polynomial parametrization and right image with isothermal {\sc Liouville} parametrization}
\label{fig:minimal_enneper}
\end{figure}
\end{example}

\begin{example}[{\sc Enneper} surface]
The following parametrization of the {\sc Enneper} minimal surface: 
\begin{align*}
{\bf x}(u, v) = \begin{pmatrix}
 -e^u \sin (v) \left(2 e^{2 u} \cos (2 v)+e^{2 u}-3\right)\\
  e^u \cos (v) \left(2 e^{2 u} \cos (2 v)-e^{2 u}+3\right)\\
3 e^{2 u} \sin (2 v)
\end{pmatrix}
\end{align*}
has the line element: ${\dif s}^2 = 9 e^{2 u} (1 + e^{2 u})^2 ({\dif u}^2 + {\dif v}^2)$.
This is an isothermal {\sc Clairaut} in $u$ line element. See right image in figure \ref{fig:minimal_enneper}.
\end{example}

\begin{example}[{Helicoid -- catenoid and their associated surfaces}]
The following parametrization of the family of minimal surfaces with parameter $t \in [0,1]$: 
\begin{align*}
{\bf x}(u, v) = \begin{pmatrix}
-e^{-u} \sin \left(\frac{\pi  t}{2}-v\right)-4 e^u \sin \left(\frac{\pi  t}{2}+v\right)\\
e^{-u} \cos \left(\frac{\pi  t}{2}-v\right)-4 e^u \cos \left(\frac{\pi  t}{2}+v\right)\\
-4 \left(u \sin \left(\frac{\pi  t}{2}\right)+v \cos \left(\frac{\pi  t}{2}\right)\right)
\end{pmatrix}
\end{align*}
has the line element: ${\dif s}^2 = \left(e^{-2 u}+16 e^{2 u}+8\right) ({\dif u}^2 + {\dif v}^2)$.
This is an isothermal {\sc Clairaut} in $u$ line element. For $t=0$ we obtain the helicoid and for $t=1$ the catenoid, see figure \ref{fig:minimal_catenoid}. Their metric stays the same and does not depend on the parameter $t$. 
\begin{figure}[hbt]
\centering
$\begin{array}{cccc}
\includegraphics[scale=0.5]{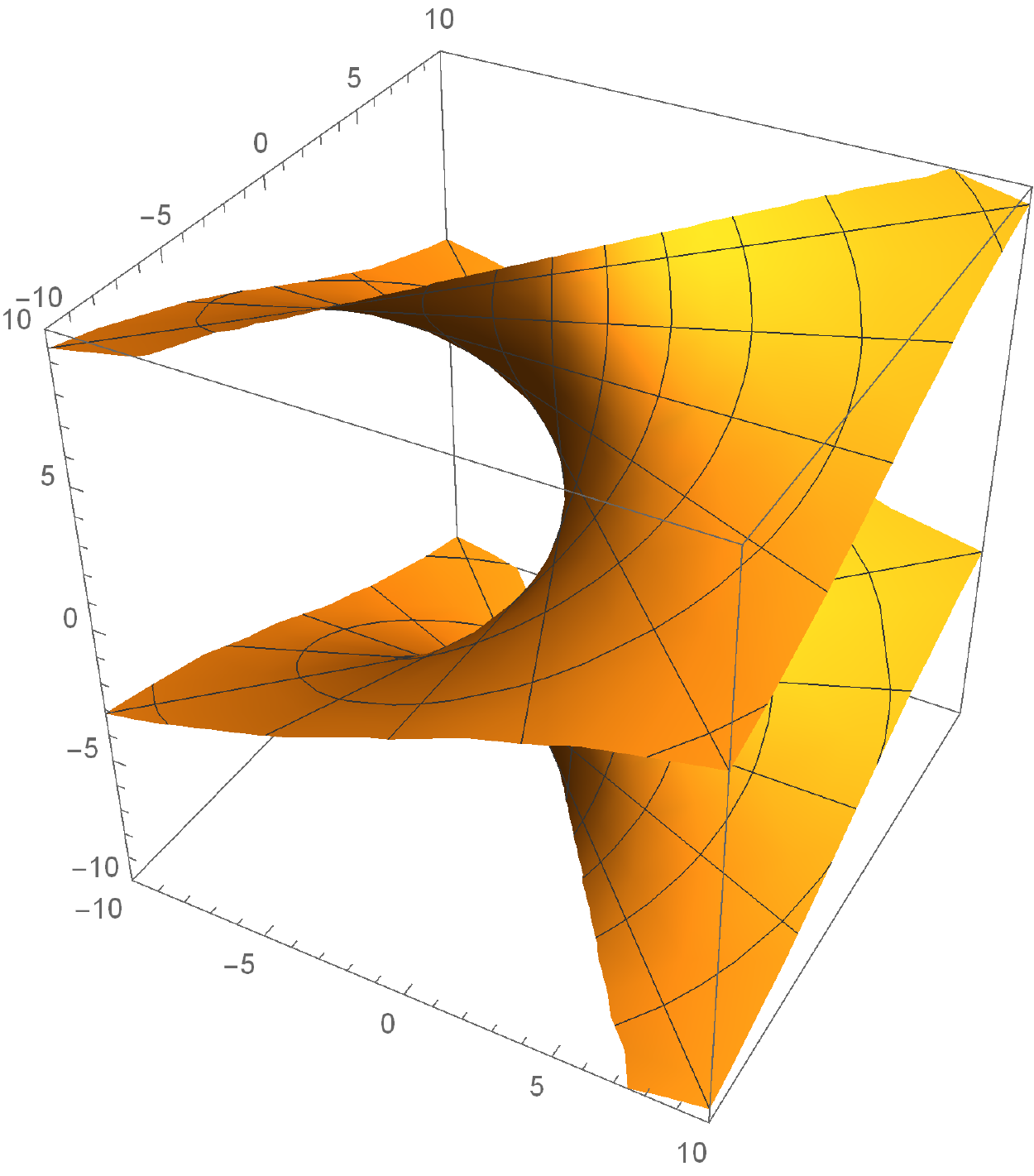} &
\includegraphics[scale=0.5]{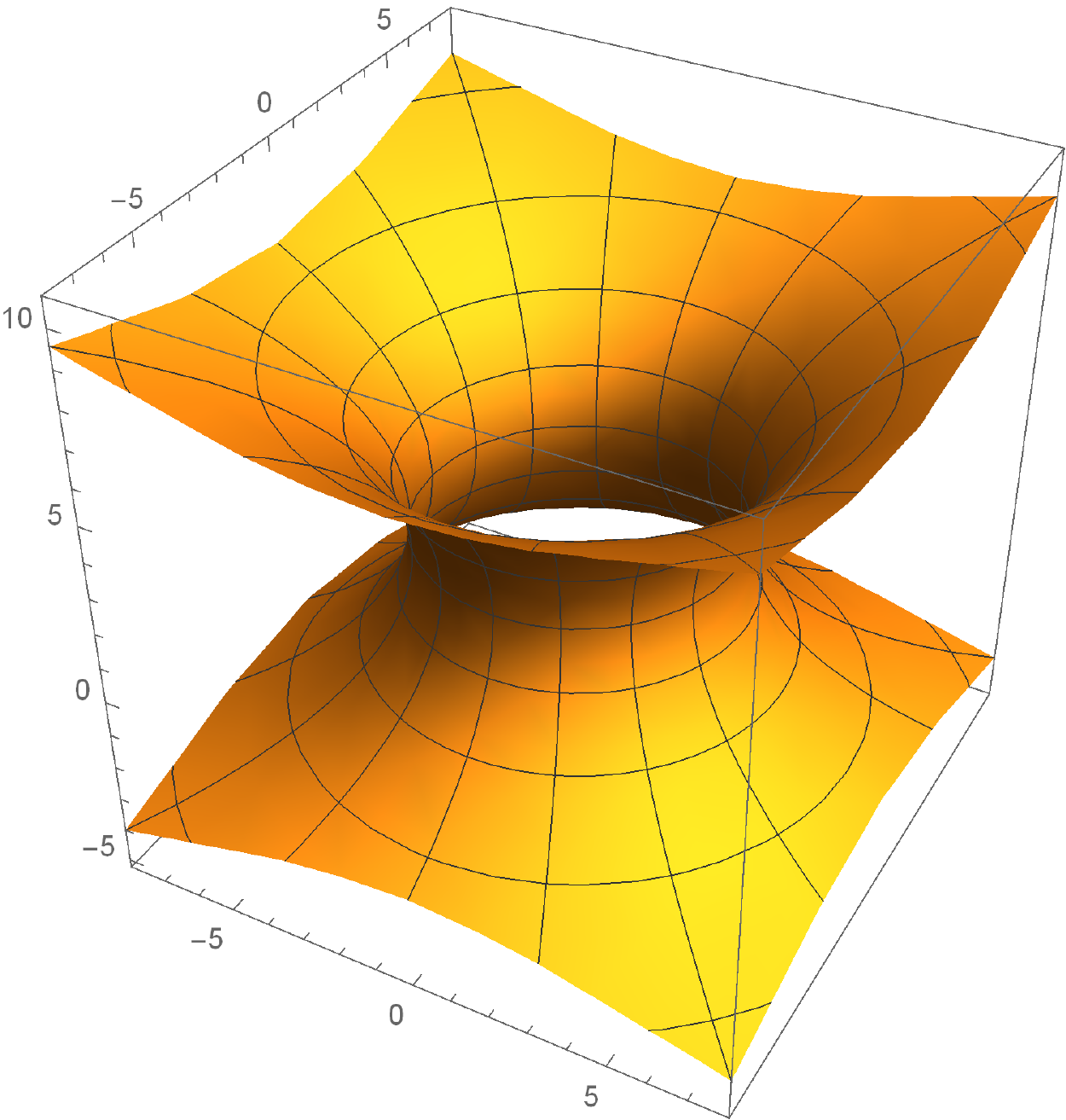}
\end{array}$
\caption{Helicoid (left image) and catenoid (right image) with isothermal {\sc Liouville} parametrization}
\label{fig:minimal_catenoid}
\end{figure}
\end{example}

\subsubsection{Quadrics}
\begin{example}[Quadrics]
Examples for quadric parametrizations are (see next section for a triaxial ellipsoid, figure \ref{fig:result}):
\begin{enumerate}
\item Standard curvature line parametrization of the orthogonal {\sc Stäckel} type, see formula (\ref{ellipsoid_scl}). The diagonals do not have the same energy in this parametrization, but the geodesic diagonals have the same length ({\sc Ivory}).
\item Isothermal {\sc Liouville} curvature line parametrization, see formula (\ref{ellipsoid_lcl}). Here the diagonals have the same energy and the geodesic diagonals have the same length ({\sc Ivory}).
\end{enumerate}
\end{example}

\section{Triaxial ellipsoid as example for quadrics}

\subsection{Standard curvature line parametrization of the triaxial ellipsoid}

In the literature (see \cite{bmu}, \cite{nyr}, \cite{val}), the authors describe how to map a triaxial ellipsoid conformally to a plane. The best paper (to my knowledge) on this matter is \cite{nyr} because it actually computes (making use of elliptic integrals) the integrals already given by {\sc Jacobi} in his ``Lectures on Dynamics''. In this article we want to go in the opposite direction and map a plane rectangle conformally to a triaxial ellipsoid in such a way that the map has an isothermal {\sc Liouville} line element. The result can be seen in the right image of figure \ref{fig:result}.
\begin{figure}[H]
\centering
$\begin{array}{cccc}
\includegraphics[scale=0.5]{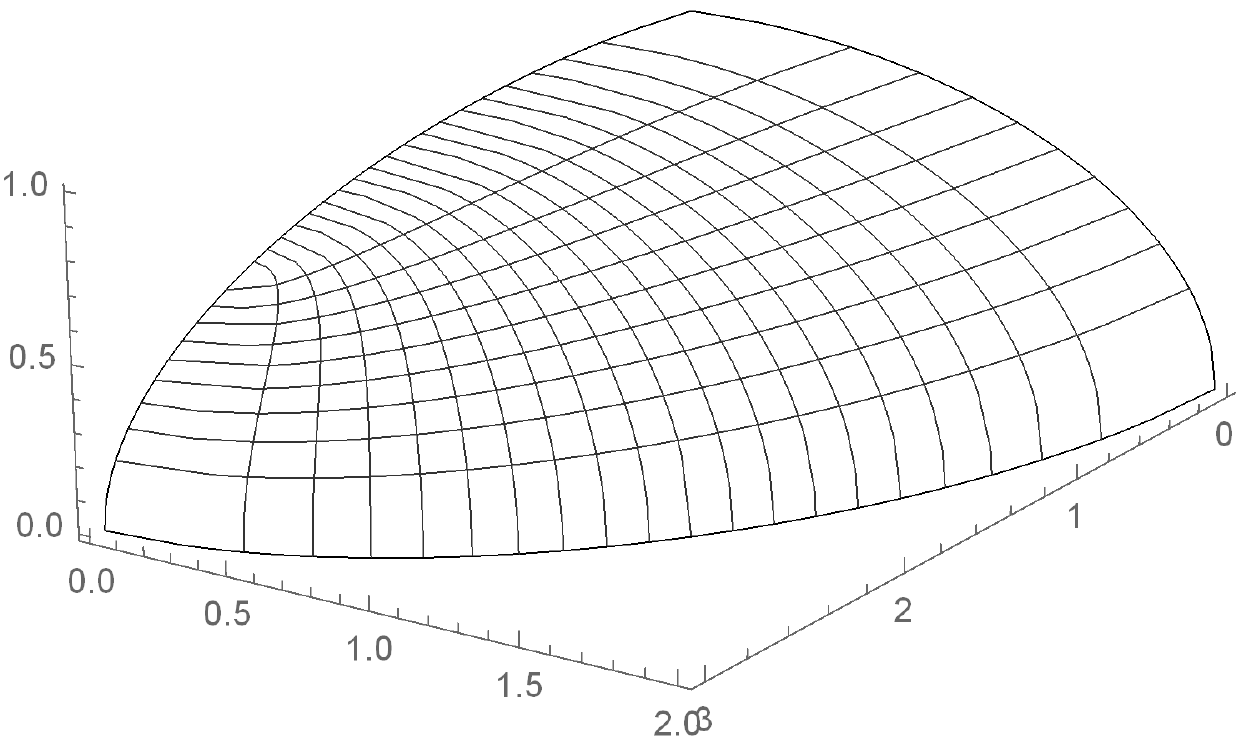} &
\includegraphics[scale=0.5]{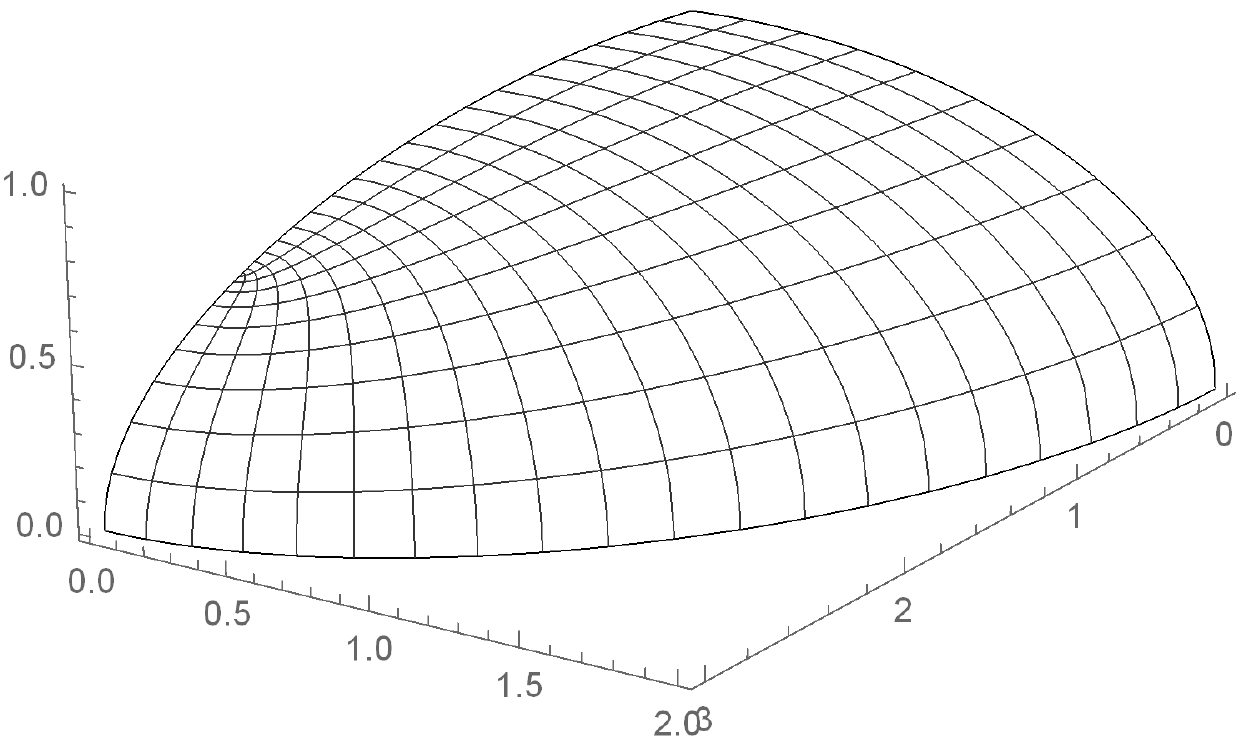}
\end{array}$
\caption{Standard curvature line (left image) and isothermal {\sc Liouville} (right image) parametrization of a triaxial ellipsoid}
\label{fig:result}
\end{figure}

We will start here with the standard curvature line parametrization of the triaxial ellipsoid with semi-axes $0 < c < b < a$:
\begin{align} \label{ellipsoid_scl} 
{\bf x}(u,v)=\left( 
\sqrt{\frac{a^2 (a^2 - u) (a^2 - v)}{(a^2 - b^2) (a^2 - c^2)}}, 
\sqrt{\frac{b^2 (b^2 - u) (b^2 - v)}{(b^2 - c^2) (b^2 - a^2)}}, 
\sqrt{\frac{c^2 (c^2 - u) (c^2 - v)}{(c^2 - a^2) (c^2 - b^2)}}
\right)^t 
\end{align}
where $0 < c^2 < v < b^2 < u < a^2$ (see left image of figure \ref{fig:result}).

The coefficients of the first fundamental form are computed as follows:
\begin{align*} 
g_{11}(u, v) &= {\bf x}_u (u, v) \cdot {\bf x}_u (u, v) = (u-v) f(u)\\
g_{12}(u, v) &= {\bf x}_u (u, v) \cdot {\bf x}_v (u, v) = g_{21}(u, v) = 0\\
g_{22}(u, v) &= {\bf x}_v (u, v) \cdot {\bf x}_v (u, v) = (u-v) (-f(v))
\end{align*}
with the function $f$ defined as:
\begin{align*} 
f(t)=\frac{1}{4} \frac{t}{(a^2 - t) (b^2 - t) (c^2 - t)}
\end{align*}

The line element of the ellipsoid is:
\begin{align} 
{\dif s}^2 = g_{11}(u, v) {\dif u}^2 + g_{22}(u, v) {\dif v}^2 = (u-v) (f(u){\dif u}^2-f(v){\dif v}^2) \label{old_ds}
\end{align}

\begin{remark}
This line element is an orthogonal {\sc Stäckel} line element, because it can be written as:
\begin{align*} 
{\dif s}^2 = (u-v) (f(u){\dif u}^2-f(v){\dif v}^2) =
\begin{vmatrix}
u f(u) & v f(v)   \\
  f(u) &   f(v)
\end{vmatrix}
\left( \frac{{\dif u}^2}{f(v)} - \frac{{\dif v}^2}{f(u)} \right)
\end{align*}
\end{remark}

\subsection{Conformal map from ellipsoid to plane}

What we want to achieve is the following isothermal {\sc Liouville} form of this line element (\ref{old_ds}):
\begin{align} 
{\dif s}^2 &= (U(x)-V(y)) ({\dif x}^2+{\dif y}^2) \label{new_ds}
\end{align}

If formulas (\ref{old_ds}) and (\ref{new_ds}) are to be the same we must have:
\begin{align*} 
       {\dif x} = \sqrt{+f(u)}{\dif u} \quad \text{and} \quad
       {\dif y} = \sqrt{-f(v)}{\dif v}
\end{align*}

By integrating we get formulas corresponding to (7) and (8) from \cite{nyr}:
\begin{align*} 
       X(u) &= \int_{b^2}^{u} \sqrt{+f(t)}{\dif t} = F_1(u)-F_1(b^2) = F_1(u) \\
       Y(v) &= \int_{c^2}^{v} \sqrt{-f(t)}{\dif t} = F_2(v)-F_2(c^2) = F_2(v)
\end{align*}
with
\begin{align*}
F_1\left(t\right)&=\frac{b^2 i}{c \sqrt{ a^2 - b^2 }} \Pi\left(n_1; \varphi_1(t) | m_1 \right)  \\ 
F_2\left(t\right)&=\frac{c^2  }{b \sqrt{ a^2 - c^2 }} \Pi\left(n_2; \varphi_2(t) | m_2 \right) 
\end{align*}
where $i=\sqrt{-1}$ and
\begin{align*}
n_1 &= 1 - \frac{b^2}{c^2} &
\varphi_1(t) &= \arcsin \left({-i c \sqrt{\frac{ t - b^2 }{\left( b^2 - c^2 \right) t  }}}\right) 
          & m_1 &= \frac{{a}^2\left( c^2 - b^2 \right) }{{c}^2\left( a^2 - b^2 \right) }\\
n_2 &= 1 - \frac{c^2}{b^2} &
\varphi_2(t) &= \arcsin \left({b \sqrt{\frac{ t - c^2 }{\left( b^2 - c^2 \right) t  }}}\right) 
          & m_2 &= \frac{{a}^2\left( b^2 - c^2 \right) }{{b}^2\left( a^2 - c^2 \right) }
\end{align*}
and the incomplete elliptic integral of the third kind is defined as follows:
\begin{align*}
\Pi(n; \varphi | m)=\int_0^{\varphi} \frac{\dif \theta}{(1 - n\sin^2 \theta)\sqrt{1 - m \sin^2 \theta}}
\end{align*}

\subsection{Differential equations}

If we plug $u=U(x)$ and $v=V(y)$ in the equation (\ref{old_ds}) of the line element of the ellipsoid we get:
\begin{align*}
{\dif s}^2 = \left(U(x)-V(y)\right) 
                                \left(f(U(x))\left(\frac{\dif U(x)}{\dif x}\right)^2 {\dif x}^2
                                -f(V(y))\left(\frac{\dif V(y)}{\dif y}\right)^2 {\dif y}^2 \right)
\end{align*}
Comparing this formula with (\ref{new_ds}) we see that the functions $U(x)$ and $V(y)$ satisfy the following differential equations:
\begin{align*}
\frac{\dif U(x)}{\dif x} = \sqrt{\frac{+1}{f(U(x))}} \quad \text{and} \quad
\frac{\dif V(y)}{\dif y} = \sqrt{\frac{-1}{f(V(y))}}
\end{align*}

\subsection{Isothermal {\sc Liouville} map from plane to ellipsoid}

We are interested in the inverse functions $U(x)$ and $V(y)$ of $X(u)$ and $Y(v)$. We proceed as follows:

We define a generalized {\sc Jacobi} amplitude $\am(n;z|m)$ as inverse function of the elliptic integral of the third kind. That means 
\begin{align*}
z &= \Pi(n; \varphi | m) \\
\am(n; z | m) &= \varphi
\end{align*}
The {\sc Jacobi} amplitude as special case can be expressed in terms of this generalized {\sc Jacobi} amplitude as $\am(z|m)=\am(0;z|m)$. With the generalized {\sc Jacobi} amplitude we can invert the elliptic integrals of the third kind and get:
\begin{align*}
U(x)&=\frac{b^2}{1-n_1 \sin ^2\left(\am \left(n_1;\frac{x c \sqrt{a^2-b^2}}{i b^2}|m_1\right)\right)} \\
V(y)&=\frac{c^2}{1-n_2 \sin ^2\left(\am \left(n_2;\frac{y b \sqrt{a^2-c^2}}{c^2}  |m_2\right)\right)}
\end{align*}
We can introduce the generalized {\sc Jacobi} elliptic function $\sn(n;z|m) = \sin(\am(n;z|m))$ and an associated function $\en(n;z|m)$ (see next section) to get:
\begin{align*}
U(x)&=\frac{b^2}{1-n_1 \sn ^2\left(n_1;\frac{x c \sqrt{a^2-b^2}}{i b^2}|m_1\right)} 
     =\frac{b^2}{\en ^2\left(n_1;\frac{x c \sqrt{a^2-b^2}}{i b^2}|m_1\right)}\\
V(y)&=\frac{c^2}{1-n_2 \sn ^2\left(n_2;\frac{y b \sqrt{a^2-c^2}}{c^2}  |m_2\right)}
     =\frac{c^2}{\en ^2\left(n_2;\frac{y b \sqrt{a^2-c^2}}{c^2}  |m_2\right)}
\end{align*}

Then the isothermal {\sc Liouville} parametrization of the ellipsoid is given by:
\begin{align} \label{ellipsoid_lcl}
{\bf x}(U(x),V(y))
\end{align}
where $0 = X(b^2) < x < X(a^2)$ and $0 = Y(c^2) < y < Y(b^2)$.

\subsection{Construction of the generalized $\sn(n;z|m)$ function}
Here I construct a series representation of the generalized {\sc Jacobi} $\sn$ function by using
the {\sc Lie} series method to invert an incomplete elliptic integral of the third kind.

\subsubsection{Elliptic integral in {\sc Jacobi} form}

The {\sc Jacobi} form of the elliptic integral of the third kind is given by:
\begin{align*}
\Pi(n ; x | m) &= \int_{0}^{x} \frac{d t}{(1 - n t^2)\sqrt{(1 - t^2)(1 - m t^2)}}
\end{align*}
Here $m = k^2$ is called the modulus and $n$ is called the characteristic.

\subsubsection{Construction of the generalized {\sc Jacobi} $\sn$ function}

Here we use the method outlined in \cite{Gro} with:
\begin{align}
f(x, y) &= y^2 - (1 - n x^2)^2 (1 - x^2) (1 - m x^2) \label{f_x_y_sn} \\
\varphi(x, y) &= \frac{1}{y} \nonumber
\end{align}
and the differential operator:
\begin{align*}
D = \frac{1}{\varphi} \frac{\partial}{\partial x} - \frac{1}{\varphi} \frac{f_x}{f_y} \frac{\partial}{\partial y} = y \frac{\partial}{\partial x} - \frac{f_x}{2} \frac{\partial}{\partial y}
\end{align*}
to construct the {\sc Lie} series:
\begin{align*}
\sn(n; u | m) = \left. e^{u D} x \right|_{(x,y)=(0,1)} \\
\sn'(n; u | m) = \left. e^{u D} y \right|_{(x,y)=(0,1)} \\
\end{align*}
where:
\begin{align*}
e^{u D} x = \sum_{j=0}^{\infty} \frac{u^j \overbrace{D(D( \cdots D}^j(x)))}{j !}
\end{align*}

\subsubsection{Inversion and evaluation of the elliptic integral of the third kind in {\sc Jacobi} form}

Let's consider:
\begin{align*}
\Pi(n ; x | m) = \int_{0}^{x} \frac{d t}{(1 - n t^2)\sqrt{(1 - t^2)(1 - m t^2)}}
\end{align*}
Now use the generalized {\sc Jacobi} $x = \sn(n; u | m)$ function:
\begin{align*}
\Pi(n ; \sn(n; u | m) | m) = \int_{0}^{\sn(n; u | m)} \frac{d t}{(1 - n t^2)\sqrt{(1 - t^2)(1 - m t^2)}}
\end{align*}
Substitute $t = \sn(n; s | m)$. We set:
\begin{align*}
\sqrt{1 - t^2}     &= \sqrt{1 - \sn^2 (n; s | m)} =: \cn(n; s | m) \\
\sqrt{1 - m t^2} &= \sqrt{1 - m \sn^2 (n; s | m)} := \dn(n; s | m) \\
\sqrt{1 - n t^2}   &= \sqrt{1 - n \sn^2 (n; s | m)} =: \en(n; s | m) 
\end{align*}
We replace in formula (\ref{f_x_y_sn}) $x = \sn (n; s | m)$ and $y = \sn' (n; s | m)$ and set $f(x, y) = 0$ to get the following identity:
\begin{align*}
\frac{dt}{ds} &= \sn' (n; s | m) = \en^2 (n; s | m) \cn(n; s | m) \dn(n; s | m) 
\end{align*}
With this we have verified the inversion:
\begin{align*}
\Pi(n ; \sn(n; u | m) | m) = \int_{0}^{u} \frac{\en^2 (n; s | m) \cn(n; s | m) \dn(n; s | m)}{\en^2 (n; s | m) \cn(n; s | m) \dn(n; s | m)} ds = \int_{0}^{u} ds = u
\end{align*}

\subsection{Acknowledgements}

I want to thank Prof. Maxim Nyrtsov for sending me his paper \cite{nyr} about the {\sc Jacobi} conformal map from ellipsoid to plane. I want to thank Albert D. Rich for his invaluable help in computing the two integrals $F_1(t)$ and $F_2(t)$. He will add these integrals to his rule based integrator, see \cite{adr}.
For numerical methods of inversion, see \cite{fuk}.

\section{Geodesics on {\sc Liouville} surfaces}

Here we study the geodesics on {\sc Liouville} surfaces.

\subsection{Geodesic equations for orthogonal surface patches}

\subsubsection{Preparations}

Let's differentiate the coefficients of the first fundamental form with respect to $u$ and $v$ and use the orthogonality of the surface patch (that means $g_{12} = {\bf x}_u \cdot {\bf x}_v = 0$):
\begin{align} \label{fff_deriv}
    \partial_u g_{11} &= \partial_u ({\bf x}_u \cdot {\bf x}_u) = 2 {\bf x}_u \cdot {\bf x}_{uu} & &
{\bf x}_u \cdot {\bf x}_{uu} = \frac{\partial_u g_{11}}{2} \\
    \partial_v g_{11} &= \partial_v ({\bf x}_u \cdot {\bf x}_u) = 2 {\bf x}_u \cdot {\bf x}_{uv} & &
{\bf x}_u \cdot {\bf x}_{uv} = \frac{\partial_v g_{11}}{2} \nonumber \\
    \partial_u g_{22} &= \partial_u ({\bf x}_v \cdot {\bf x}_v) = 2 {\bf x}_v \cdot {\bf x}_{uv} & &
{\bf x}_v \cdot {\bf x}_{uv} = \frac{\partial_u g_{22}}{2} \nonumber \\
    \partial_v g_{22} &= \partial_v ({\bf x}_v \cdot {\bf x}_v) = 2 {\bf x}_v \cdot {\bf x}_{vv} & &
{\bf x}_v \cdot {\bf x}_{vv} = \frac{\partial_v g_{22}}{2} \nonumber \\
0 = \partial_u g_{12} &= \partial_u ({\bf x}_u \cdot {\bf x}_v) = {\bf x}_v \cdot {\bf x}_{uu} + {\bf x}_u \cdot {\bf x}_{uv} & &
{\bf x}_v \cdot {\bf x}_{uu} = - \frac{\partial_v g_{11}}{2} \nonumber \\
0 = \partial_v g_{12} &= \partial_v ({\bf x}_u \cdot {\bf x}_v) = {\bf x}_v \cdot {\bf x}_{uv} + {\bf x}_u \cdot {\bf x}_{vv} & &
{\bf x}_u \cdot {\bf x}_{vv} = - \frac{\partial_u g_{22}}{2} \nonumber 
\end{align}

\subsubsection{Geodesic equations}

Now we consider (see \cite{Lew}) a curve $\alpha(t) = {\bf x}(u(t), v(t))$ on a surface and differentiate it twice with respect to $t$. The first derivative is $\dot{\alpha}(t) = {\bf x}_u \dot{u} + {\bf x}_v \dot{v}$ and the second one is:
\begin{align*}
\ddot{\alpha}(t) &= {\bf x}_u \ddot{u} + \dot{u}({\bf x}_{uu} \dot{u} + {\bf x}_{uv} \dot{v})
                  + {\bf x}_v \ddot{v} + \dot{v}({\bf x}_{uv} \dot{u} + {\bf x}_{vv} \dot{v}) \\
                 &= {\bf x}_u \ddot{u} + {\bf x}_v \ddot{v}
                  + {\bf x}_{uu} \dot{u}^2 + 2 {\bf x}_{uv} \dot{u} \dot{v} + {\bf x}_{vv} \dot{v}^2  
\end{align*}
If $\alpha(t)$ is a geodesic, then it is normal to the surface, that means:
\begin{align*}
\ddot{\alpha}(t) \cdot {\bf x}_u = 0 \quad \text{and} \quad\ddot{\alpha}(t) \cdot {\bf x}_v = 0
\end{align*}
These two conditions lead to the following two geodesic equations (by using $g_{12} = {\bf x}_u \cdot {\bf x}_v = 0$):
\begin{align*}
{\bf x}_u \cdot {\bf x}_u \ddot{u} + {\bf x}_u \cdot {\bf x}_{uu} \dot{u}^2 + 2 {\bf x}_u \cdot {\bf x}_{uv} \dot{u} \dot{v} + {\bf x}_u \cdot {\bf x}_{vv} \dot{v}^2 &= 0 \\  
{\bf x}_v \cdot {\bf x}_v \ddot{v} + {\bf x}_v \cdot {\bf x}_{uu} \dot{u}^2 + 2 {\bf x}_v \cdot {\bf x}_{uv} \dot{u} \dot{v} + {\bf x}_v \cdot {\bf x}_{vv} \dot{v}^2 &= 0   
\end{align*}
We can use the previous results (\ref{fff_deriv}) to replace and get:
\begin{align*}
g_{11} \ddot{u} + \frac{\partial_u g_{11}}{2} \dot{u}^2 + 2 \frac{\partial_v g_{11}}{2} \dot{u} \dot{v} - \frac{\partial_u g_{22}}{2} \dot{v}^2 &= 0 \\  
g_{22} \ddot{v} - \frac{\partial_v g_{11}}{2} \dot{u}^2 + 2 \frac{\partial_u g_{22}}{2} \dot{u} \dot{v} + \frac{\partial_v g_{22}}{2} \dot{v}^2 &= 0   
\end{align*}
We can divide the first equation by $g_{11}$ and the second by $g_{22}$ to get:
\begin{align*}
\ddot{u} + \frac{\partial_u g_{11}}{2 g_{11}} \dot{u}^2 + 2 \frac{\partial_v g_{11}}{2 g_{11}} \dot{u} \dot{v} - \frac{\partial_u g_{22}}{2 g_{11}} \dot{v}^2 &= 0 \\  
\ddot{v} - \frac{\partial_v g_{11}}{2 g_{22}} \dot{u}^2 + 2 \frac{\partial_u g_{22}}{2 g_{22}} \dot{u} \dot{v} + \frac{\partial_v g_{22}}{2 g_{22}} \dot{v}^2 &= 0   
\end{align*}
By using the {\sc Christoffel} symbols these geodesic equations can be written as:
\begin{align*}
\ddot{u} + \Gamma^1_{11} \dot{u}^2 + 2 \Gamma^1_{12} \dot{u} \dot{v} + \Gamma^1_{22} \dot{v}^2 &= 0 \\  
\ddot{v} + \Gamma^2_{11} \dot{u}^2 + 2 \Gamma^2_{12} \dot{u} \dot{v} + \Gamma^2_{22} \dot{v}^2 &= 0   
\end{align*}
For checking the geodesic equations we use the following form:
\begin{align} \label{geodesic_eqn_check}
2 g_{11} \ddot{u} + \partial_u g_{11} \dot{u}^2 + 2 \partial_v g_{11} \dot{u} \dot{v} - \partial_u g_{22} \dot{v}^2 &= 0 \\  
2 g_{22} \ddot{v} - \partial_v g_{11} \dot{u}^2 + 2 \partial_u g_{22} \dot{u} \dot{v} + \partial_v g_{22} \dot{v}^2 &= 0 \nonumber  
\end{align}

\subsection{Geodesics on isothermal {\sc Liouville} surfaces}

We repeat here the explanation from \cite{dar} how to arrive at the differential equations of the geodesics of an isothermal {\sc Liouville} surface.

Start with the line element of the isothermal {\sc Liouville} surface:
\begin{align*}
{\dif s}^2 = (U + V)({\dif u}^2 + {\dif v}^2)
\end{align*}
In a first step rewrite it as product of sums of squares:
\begin{align*}
{\dif s}^2 = (\sqrt{U - a}^2 + \sqrt{a + V}^2)({\dif u}^2 + {\dif v}^2)
\end{align*}
Then this can be written as a sum of squares:
\begin{align*}
{\dif s}^2 = (\sqrt{U - a}{\dif u} + \sqrt{a + V}{\dif v})^2 + (\sqrt{U - a}{\dif v} - \sqrt{a + V}{\dif u})^2
\end{align*}
With:
\begin{align*}
{\dif t} &= \sqrt{U - a}{\dif u} + \sqrt{a + V}{\dif v} \\
{\dif t_1} &= \frac{{\dif u}}{\sqrt{U - a}} - \frac{{\dif v}}{\sqrt{a + V}}
\end{align*}
we get geodesic parallel coordinates:
\begin{align*}
{\dif s}^2 = {\dif t}^2 + (U-a)(a + V){\dif t_1}^2
\end{align*}
And here we see that the geodesics are given by $t_1 = \text{const.}$ that means ${\dif t_1} = 0$:
\begin{align*}
0 = {\dif t_1} = \frac{{\dif u}}{\sqrt{U - a}} - \frac{{\dif v}}{\sqrt{a + V}}
\end{align*}
or equivalently:
\begin{align*}
\frac{{\dif u}}{\sqrt{U - a}} = \frac{{\dif v}}{\sqrt{a + V}}
\end{align*}
Then we have:
\begin{align*}
{\dif s} = {\dif t} &= \sqrt{U - a} {\dif u} + \sqrt{a + V} {\dif v} \\
                         &= \frac{(U - a) {\dif u}}{\sqrt{U - a}} + \frac{(a + V) {\dif v}}{\sqrt{a + V}} \\
                         &= \frac{U {\dif u}}{\sqrt{U - a}} + \frac{V {\dif v}}{\sqrt{a + V}}
\end{align*}
and after that:
\begin{align*}
\frac{{\dif t}}{U + V} = \frac{{\dif u}}{\sqrt{U - a}} = \frac{{\dif v}}{\sqrt{a + V}}
\end{align*}
This gives finally a system of first order differential equations for the geodesics:
\begin{align*}
\frac{{\dif u}}{{\dif t}} &= \frac{\sqrt{U - a}}{U + V} \\
\frac{{\dif v}}{{\dif t}} &= \frac{\sqrt{a + V}}{U + V} 
\end{align*}
These expressions satisfy the geodesic equations (\ref{geodesic_eqn_check}).

\section{Higher dimensions}

\subsection{Isothermal {\sc Liouville} maps}

In higher dimensions ($n \ge 3$) we have (see \cite{ms}):

\begin{theorem}[{\sc Liouville}'s theorem on conformal mappings]
Let ${\bf x} : O \to {\bf x}(O)$ be a one-to-one $C^n$ conformal map, where $O \subset \R^n$ for $n \ge 3$ is open. Then ${\bf x}$ is a composition of isometries, dilations and inversions.
\end{theorem}

This generalized theorem states that every conformal map ${\bf x}$ in $\R^n$ for $n \ge 3$ is a composition of Möbius transformations. Therefore the isothermal {\sc Liouville} manifolds in higher dimensions than $2$ are somewhat restricted.

\subsection{Main theorem for $n$ dimensions}

The main theorem \ref{main_thm} also holds for higher dimensions and the proof is similar to the $2$-dimensional case, therefore we only state it here:

\begin{theorem}[Main theorem for general $n$]
Consider two points $P^0(p^0_1,p^0_2,\dots,p^0_n)$ and $P^1(p^1_1,p^1_2,\dots,p^1_n)$ in an open convex subset $O \subset \R^n$. Then these points uniquely determine an $n$-rectangle in $O$ (possibly degenerated). Now consider a map ${\bf x} : O \to {\bf x}(O)$. If and only if this map ${\bf x}$ has the following orthogonal {\sc Liouville} line element: 
\begin{align*}
{\dif s}^2 = \sum_{k=1}^{n} \left( \sum_{i=1}^{n} U_{i k}(u_i)\right) {\dif u}_k^2
\end{align*}
the images (under the map ${\bf x}$) of the diagonals in the $n$-rectangle determined by the points $P^0$ and $P^1$ have the same energy.
\end{theorem}
It should be noted that there is no restriction in choosing the two points $P^0$ and $P^1$ in the domain $O$ of definition.

\end{document}